\documentclass[10pt,a4paper]{article}

\usepackage[utf8]{inputenc}
\usepackage{amsmath,amsfonts,amssymb,mathtools,mathrsfs,%
enumitem}
\usepackage{amsthm}
\usepackage{optidef}
\usepackage{comment}
\usepackage{hyperref}
\usepackage[backend=biber,
sorting=nty,
maxnames=99
]{biblatex} 
\usepackage[margin=1in]{geometry}
\usepackage{authblk}

\newcommand{\easyineqslemma}{\cite[Lem.~1]{OurSufficientPaper}}

\newtheorem{theorem}{Theorem}
\newtheorem{proposition}[theorem]{Proposition}

\newtheorem{lemma}[theorem]{Lemma}
\newtheorem{corollary}[theorem]{Corollary}
\theoremstyle{remark}
\newtheorem{remark}[theorem]{Remark} 
\theoremstyle{definition}

\newtheorem{definition}[theorem]{Definition}
\newtheorem{assumpt}[theorem]{Assumption}

\usepackage[dvipsnames]{xcolor}
\definecolor{dkgreen}{rgb}{0,0.4,0}
\definecolor{dkblue}{rgb}{0.0,0,0.8}

\definecolor{dkgreen}{rgb}{0,0.4,0}
\definecolor{dkred}{rgb}{0.8,0.0,0}

\title{%
Young measure relaxation gaps for controllable systems with smooth state constraints
}

\date{}

\newcommand{\R}{\mathbb R}

\DeclareMathOperator{\dist}{\operatorname{dist}}

\newcommand{\domain}{\Omega}
\newcommand{\T}{T}
\newcommand{\controls}{U}
\newcommand{\g}{g}
\newcommand{\f}{f}
\newcommand{\F}{F}
\newcommand{\U}{U_\Box}
\newcommand{\cornersU}{U_{::}}
\newcommand{\initial}{{\mathbf{x}_0}}
\newcommand{\normalL}{L}
\newcommand{\nonconvexL}{{L^{\vee\!\vee}}}
\newcommand{\constantL}{{L^-}}
\newcommand{\convexL}{{{L^\vee}}}
\newcommand{\minarcs}{M_{\mathrm{y}}}
\newcommand{\mincurves}{M_{\mathrm{c}}}

\newcommand{\minoccupation}{M_{\mathrm{o}}}

\newcommand{\var}[1]{\widehat{#1}}
\newcommand{\variational}{V}
\newcommand{\variationalrelaxed}{\widehat{V}}

\newcommand{\ballbox}{\bar C}%
\newcommand{\ballboxsmall}{\bar c}%

\newcommand{\destination}{X}

\newcommand{\ad}{a_\delta}

\newcounter{Ccnt}
\makeatletter
\newcommand\C[1]{%
\@ifundefined{C-#1}%
  {\stepcounter{Ccnt}\expandafter\xdef\csname C-#1\endcsname{\arabic{Ccnt}}}%
  {}%
c_{\csname C-#1\endcsname}}
\makeatother 

\begin{document}

\author[1]{Nicolas Augier}
\author[1,2]{Milan Korda}
\author[2,3]{Rodolfo Rios-Zertuche}
\affil[1]{\footnotesize 
LAAS-CNRS, Toulouse, France
}
\affil[2]{\footnotesize 
Faculty of Electrical Engineering, Czech Technical University in Prague, Czechia
}
\affil[3]{\footnotesize Department of Mathematics and Statistics, UiT The Arctic University of Norway}

\maketitle

\begin{abstract}
In this article, we tackle the problem of the existence of a gap corresponding to 
Young measure relaxations for state-constrained optimal control problems. We provide a counterexample proving that a gap may occur in a very regular setting, namely for  a smooth controllable system state-constrained to the  closed unit ball, provided that the Lagrangian density (i.e., the running cost) is non-convex in the control variables. The example is constructed in the setting of sub-Riemannian geometry with the core ingredient being an unusual admissible curve that exhibits a certain form of resistance to state-constrained approximation. Specifically, this curve cannot be approximated by neighboring admissible curves while obeying the state constraint due to the intricate nature of the dynamics near the boundary of the constraint set. This example therefore demonstrates the impossibility of  Filippov--Wazewski type approximation in the presence of state constraints.
Our example also presents an occupation measure relaxation gap.
\end{abstract}
\tableofcontents

\section{Introduction}
\label{sec:intro}
\subsection{Main question}
A problem of theoretical and practical interest is: 
\begin{equation}\label{eq:question}
\text{When is $\mincurves=\minarcs$? }
\end{equation}
Here, $\mincurves$ denotes the minimal cost of an optimal control problem, with the infimum taken over measurable controls, and $\minarcs$ denotes its relaxation using Young measures; we will define these shortly.   

As we shall see, the Young measure relaxation is equivalent to the convexification of the associated differential inclusion (see Lemma \ref{lem:convexequiv}). 
Understanding when state constraints impede the approximation, by admissible curves in the interior, of a trajectory of the convexified differential inclusion that lies on the boundary, is a question of great theoretical importance. 
It has been tackled in \cite{frankowska2000filippov, bettiolfrankowska,V93,OurSufficientPaper} providing many precise estimates and sufficiency conditions for equality \eqref{eq:question} to hold. 
This question is important for motion planning of a multitude of state-constrained control systems, including even one of the most natural controllable systems, namely, the sub-Riemannian ones with linear control dependence.

On the more applied side, there exist algorithms based on moment-sums-of-squares hierarchies aiming to compute $\mincurves$  through the so-called occupation measure relaxation, whose optimal value coincides with $\minarcs$, and answering question \eqref{eq:question} would amount to understanding their applicability, as we explain at the end of this section.
In our paper \cite{OurSufficientPaper}, we give sufficient conditions for the absence of occupation measure relaxation gaps, which are shown to be equivalent to equality \eqref{eq:question}.
In the present work, we give  examples of systems that exhibit such a relaxation gap in circumstances that otherwise seem very natural. 
We will describe these examples in Section \ref{sec:examples}.

To clarify the meaning of question \eqref{eq:question} and to present our findings effectively, we will begin by outlining the problems of interest, along with the necessary notations and terminology.
Throughout, ``smooth'' will mean $C^\infty$.

Let $\domain$ be a subset of $\R^d$, and choose an initial point $\initial\in\domain$, a target set $\destination\subset\domain$ which will give the endpoint constraint, a fixed amount of time $\T>0$ and  a set of controls $\controls\subset\R^d$. Let also $\normalL\colon\R\times\domain\times\controls\to\R$ be a function that will serve as Lagrangian density or running cost, $\f\colon\R\times\domain\times\controls\to\R^d$ a map that will define the controlled vector field, and $\g\colon\R^d\to\R$ a function giving the final cost.

The original (classical) \emph{Bolza problem} that we want to consider and eventually relax is that of finding the minimal cost $\mincurves=
                \mincurves^{\normalL,\g,\controls}({\domain},\destination)$ given by
\begin{customopti}%
                {inf}{u}
                {\int_{0}^T \normalL(t,\gamma(t),u(t))\,dt+\g(\gamma(T))}%
                {\label{min:orig}}%
                {\mincurves=}
                \addConstraint{u\colon [0,\T]\to \controls \;\text{measurable}}
                \addConstraint{\gamma(t)=\initial+\int_0^tf(s,\gamma(s),u(s))\,ds\in\domain,\;t\in[0,\T]}
                \addConstraint{%
                \gamma(\T)\in\destination.}
\end{customopti}
Here we are minimizing the sum of the integral of $\normalL$ and the terminal cost given by $\g$ over all controls $u$ that determine admissible curves $\gamma$ in the domain $\domain$ that start at $\initial$ and end in the set $\destination$. 
These controls $u$ constitute the \emph{set of contenders} or \emph{feasible set}, of the problem; we will use these terms interchangeably.
When $g=0$, this is known as a \emph{Lagrange problem}, and when instead $L=0$, it is a \emph{Mayer problem}. 
It is well known that minimizers may fail to exist for the problem \eqref{min:orig} even under coercivity assumptions, as the lack of convexity may induce oscillatory behavior in curves approaching the minimal cost \cite{pedregal2}.

We will sometimes abbreviate the notation $\mincurves^{\normalL,\g,\controls}(\domain,\destination)$ by omitting the parameters that are clear from the context; for example, if it were clear that we were working with $\normalL$ and $\g$, we would instead write $\mincurves^{\controls}(\domain,\destination)$. The main part of this notation, that will not change, is $\mincurves$.

A \emph{Young measure} is a family $(\nu_t)_{t\in I}$ of Borel probability measures on the set of controls $U$ indexed by a parameter $t$ in an interval $I\subset \R$ and such that, for every Borel subset $A\subset U$, $t\mapsto\nu_t(A)$ is measurable, and  for every $x\in \Omega$ and every $t\in I$, the function $u\mapsto f(t,x,u)$ is $\nu_t$ integrable, that is,
\[\int_{U}\|f(t,x,u)\|d\nu_t(u)<+\infty.\] 
The relaxation we consider is the \emph{Young measure relaxation}
$\minarcs=
    \minarcs^{\normalL,\g,\controls}({\domain},\destination)$, namely,
\begin{customopti}%
    {inf}{(\nu_t)_{t\in[0,\T]}}
    {\int_{0}^\T \int_{\controls} \normalL(t,\gamma(t),u)\,d\nu_t(u)\,dt+\g(\gamma(\T))}{\label{min:relaxed}}{\minarcs=
    }
    \addConstraint{(\nu_t)_t \text{ a Young measure}}
    \addConstraint{\gamma(t)=\initial+\int_0^t\int_{U}f(s,\gamma(s),u)\,d\nu_s(u)\,ds\in\domain,\quad t\in[0,\T]}
    \addConstraint{%
    \gamma(\T)\in\destination.}
\end{customopti}
Here the contenders are the Young measures $(\nu_t)_{t\in[0,T]}$ with the additional constraint that they induce admissible integral curves $\gamma$ contained in $\domain$, starting at $\initial$ and ending in $\destination$. Unlike the controls $u$ contending in problem \eqref{min:orig}, the Young measures $(\nu_t)_t$ in \eqref{min:relaxed} naturally encode highly oscillatory behavior. An important motivation to use relaxed controls is that minimizers for this problem are known to exist under relatively mild conditions, such as bounds of the form $c(|u|^p-1)\leq L(t,x,u)\leq C(|u|^p+1)$ with $c,C>0$ and $p>1$ \cite[Th.~4.6]{pedregal2}, which are trivially verified if $U$ is compact and $L$ is continuous; see also \cite{pedregal,vinter-book,warga2014optimal} for other existence results.
Since a feasible control $u$ for problem \eqref{min:orig} is represented by the feasible Young measure $\nu_t=\delta_{u(t)}$, introducing the Young measure relaxation makes the infimum  $\minarcs\leq \mincurves$. However, it is possible to have a positive difference in some situations; this difference $\mincurves-\minarcs$ is known as a \emph{relaxation gap}.

We will be interested in identifying some situations where $\mincurves>\minarcs$, that is, where the gap $\mincurves-\minarcs$ is strictly positive.  It was predicted in \cite{palladino2014minimizers} that in this case the minimizers would be abnormal trajectories, in the terminology usually associated with the Pontryagin Maximum Principle. There are two known mechanisms for the appearance of such gaps: \begin{itemize}
    \item Reachability. Sometimes the target set is only reachable using Young measures and there are no curves $\gamma$ induced by ordinary controls $u$ joining the initial point to the target set, so we have $\minarcs<+\infty=\mincurves$, and the minimizing Young measure is an abnormal trajectory. Examples of this type were found in \cite{WARGA197541,palladino2014minimizers}.

    \item Failure of the constrained Filippov--Wa\v zewski approximation. In this case, the state constraints may cause the appearance of minimizing Young measures that cannot be approximated by curves as in the Filippov--Wa\v zewski theorem \cite{frankowska2000filippov}.

    To the best of our knowledge, no examples of this type have been found to date, except in rather contrived settings (see e.g.~\cite[\S1]{korda2022gap}). We will give examples of this type in a very regular setting, %
    where the occurrence of the relaxation gap may come as a surprise.  
\end{itemize}

\paragraph{Occupation measure relaxation.} 
A motivation to study the Young measure relaxation \eqref{min:relaxed} is that its optimal cost is equal to the optimal cost of another relaxation, called \emph{occupation measure relaxation} $\minoccupation$; in other words, we have $\minoccupation=\minarcs$ (see \easyineqslemma{} for details). This relaxation reads as follows:
\begin{customopti}
        {inf}{\mu,\mu_\partial}
    { \int_{[0,T]\times\Omega\times \controls} \normalL(t,x,u)\,d\mu(t,x,u)+\int_{\Omega}g(x)d\mu_\partial(x)}{\label{min:occupation}}{\hspace{-3cm}M_{\mathrm o}=
    }
    \addConstraint{\mu \text{ a %
    positive Borel measure on } [0,T]\times\domain\times\controls}
    \addConstraint{\mu_\partial \text{ a positive Borel measure on } X}
    \addConstraint{\int_{[0,T]\times\domain\times\controls}\|f(t,x,u)\|d\mu+\int_{\destination}\|x\|\,d\mu_\partial(x)<+\infty}
    \addConstraint{\text{and, for all $\phi\in C^\infty(\Omega)$}}
    \addConstraint{\int_{[0,\T]\times\domain\times\controls}\frac{\partial\phi}{\partial t}+\frac{\partial\phi}{\partial x}\cdot f(t,x,u)\,d\mu
    +\phi(0,\initial)-\int_\controls \phi(\T,x)d\mu_\partial(x)=0.}
\end{customopti}
The convexity properties of the set of relaxed occupation measures (i.e., measures feasible in~\eqref{min:occupation}) allow for turning the highly-nonlinear problem \eqref{min:orig} into a linear program %
for which there exist computationally tractable solution schemes \cite{augier2024symmetry,korda2018moments,OCP08}. These schemes rely on the the so-called moment-sums of squares hierarchies \cite{OCP08}, which are applicable when all the data of the problem are semi-algebraic.  The main focus of this paper is thus of great significance to the effectiveness of the algorithms described in \cite{OCP08,korda2018moments}, which compute $\minoccupation$ while aiming to find $\mincurves$.
Our counterexample shows that the applicability of algorithms relying on this relaxation fails in some cases.

\subsection{Main contributions}
\label{sec:examples}

In this paper we find several instances of gaps $\mincurves-\minarcs>0$ arising from failures of the Filippov-Wa\v zewski approximation that arise in highly regular settings and are caused solely by the interaction of the dynamics and the state constraint set.

As we explained above, to the best of our knowledge, examples of this type were only known in situations that seemed quite forced and artificial. 

While our main result is the example described below in Section \ref{sec:lagrangeintro}, we also include, for completeness, several variations whose relevance we will explain together with their statements below. Finally, we present also a statement formalizing what we mean by a failure of the Filippov-Wa\v zewski approximation.

\subsubsection{A gap in a Lagrange problem}\label{sec:lagrangeintro}
We expose here a new gap phenomenon which is entirely caused by the interplay of the dynamics with the state constraint $\Omega$, despite the considered control systems enjoy nice controllability properties. %

We are interested in building a family of examples exhibiting a relaxation gap on circumstances that seem natural: 
\begin{enumerate}[label=P\arabic*.,ref=P\arabic*]
    \item \label{it:first}The controlled vector fields are of the form $\f(t,x,u)=u_1f_1(x)+u_2f_2(x)$, where the couple of vector fields $(f_1,f_2)$ is smooth and Lie bracket generating, i.e., $f_1$ and $f_2$, together with the nested Lie brackets
    \[[f_1,f_2],\quad[f_1,[f_1,f_2]],\quad [f_1,[f_1,[f_1,f_2]]],\quad \dots,\quad \underbrace{[f_1,[f_1,[f_1,\dots,[f_1}_{d-2},f_2]]\dots],\vspace{-.3cm}\] 
    span $\R^d$ at every point.
    The associated sub-Riemannian structure induced by $f_1$ and $f_2$ is thus $C^\infty$, free, equi-regular, of rank 2; its  step is equal to $d-2$ and its growth vector is $(2,1,1,\dots,1)$ (see \cite{ABB20} for the definitions).
\item \label{it:stateconstraint} The state constraint $\overline\domain$  is the closure of an open ball $\domain=B(0,1)$ and the initial point is in its interior, $\initial \in\domain$.
\item \label{it:targetconstraint}
There is an open set $V$ such that the target set can be taken to be any nonempty subset $X\subset V\cap\overline\Omega$. 
\label{it:last}
    \end{enumerate}
\begin{theorem}[Gaps may occur in a Lagrange problem]
\label{thm:example}
Let $\Omega$ be the ball in $\R^d$, $d\ge 4$, $X\subset \domain$  a target set satisfying \ref{it:targetconstraint}, and $U$ be any subset of $\R^2$ containing $[-1,1]\times[-1,1]\subseteq U$. There is a Lagrangian density $\normalL\in C^\infty(\R^d \times U)$, a controlled vector field $\f$ induced by a rank two sub-Riemannian structure, a point $\initial\in \domain$ with the  properties \ref{it:first}--\ref{it:last} above, and (taking also $g=0$),
\begin{gather*}
    \mincurves(\overline{\domain},X) >\minarcs(\overline{\domain},X),\\
    \mincurves(\R^d,X)=\minarcs(\R^d,X).
\end{gather*}
\end{theorem}

This result will be proven in Section \ref{sec:regularityL}, where it is restated as Theorem \ref{thm:smoothlagrangian}. 
Note that the Lagrangian density $L$ is necessarily non-convex in the control variable, as it would otherwise contradict \cite[Theorem 2.3]{V93} (see also \cite[Theorem 2]{OurSufficientPaper}); it will be built as a four-well potential. 
To simplify the exposition, the proof of Theorem \ref{thm:example} will first consider a fixed endpoint and will  extend it to a set $X$ satisfying \ref{it:targetconstraint}; see Section \ref{sec:role_in_tar}.%

\subsubsection{A gap in a Bolza problem without terminal constraints}
Since the target set $X$ in Theorem \ref{thm:example} is open, we show in Section \ref{sec:targetset} that its role in the  appearance of the gaps of Theorems \ref{thm:example} and \ref{thm:secondexample} is not important. In that section \ref{sec:targetset} we present Theorem \ref{gap_withoutX}, in which the constraint that trajectories end in $X$ has been replaced with a penalization  with a suitable terminal cost $g$ while keeping the relaxation gaps. This contrasts with the situation of the example in \cite{palladino2014minimizers}, in which the initial and endpoint constraints seem to play a crucial role.

\begin{theorem}[Gaps may occur in a Bolza problem without terminal constraints]
\label{thm:example4}
Let $\Omega$ be the ball in $\R^d$, $d\ge 4$, and $U$ be any subset of $\R^2$ containing $[-1,1]\times[-1,1]\subseteq U$. There are a Lagrangian density $L$, a smooth terminal cost $g$,  a controlled vector field $\f$ induced by a rank two sub-Riemannian structure, a point $\initial\in \domain$, and (taking also $X=\overline \domain$),
\begin{gather*}
    \mincurves(\overline{\domain},\overline\domain) >\minarcs(\overline{\domain},\overline\domain),\\
    \mincurves(\R^d,\R^d)%
    =\minarcs(\R^d,\R^d).
\end{gather*}
\end{theorem}
This will be proved in Section \ref{sec:targetset}.

\subsubsection{Gaps in Mayer problems}

Next we give examples of gaps in of Mayer problems, which have no running cost, $L=0$. While the set of controls $U$ is convex, in these examples the vector fields $\f$ are \emph{not} associated to a sub-Riemannian structure and, for $x\in\overline\Omega$, the set $f(x,U)$ is not convex.

The following result comes in contrast with the example given by Palladino--Vinter \cite[Sec.~3]{palladino2014minimizers}, where they find a gap $\mincurves(\R^d,X)>\minarcs(\R^d,X)$ in a state-unconstrained situation with a closed target set $X$, and the gap can be explained by reachability issues.

\begin{theorem}[Gaps may occur in a Mayer problem with terminal constraints]
\label{thm:example2}
Let $\Omega$ be the ball in $\R^d$, $d\ge 5$, $X\subset \domain$  a target set satisfying \ref{it:targetconstraint}, and $U$ be any subset of $\R^2$ containing $[-1,1]\times[-1,1]\subseteq U$. There is a smooth terminal cost $g$, a controlled vector field $\f$, a point $\initial\in \domain$, and (taking also $L=0$),
\begin{gather*}
    \mincurves(\overline{\domain},X) >\minarcs(\overline{\domain},X),\\
    \mincurves(\R^d,X)=\minarcs(\R^d,X).
\end{gather*}
\end{theorem}
This will be proved in Section \ref{sec:vinterlike}.

\begin{theorem}[Gaps may occur in a Mayer problem without terminal constraints]
\label{thm:example3}
Let $\Omega$ be the ball in $\R^d$, $d\ge 5$, and $U$ be any subset of $\R^2$ containing $[-1,1]\times[-1,1]\subseteq U$. There is a smooth terminal cost $g$, a controlled vector field $\f$, a point $\initial\in \domain$, and (taking also $L=0$ and $X=\overline \domain$ or $X=\R^d$),
\begin{gather*}
    \mincurves(\overline{\domain},\overline\domain) >\minarcs(\overline{\domain},\overline\domain),\\
    \mincurves(\R^d,\R^d)%
    =\minarcs(\R^d,\R^d).
\end{gather*}
\end{theorem}

This will be proved in Section \ref{sec:targetset}.

\subsubsection{Filippov--Wa\v zewski breakdown phenomenon at the boundary}
\label{sec:FWintro}

Denote by $\mathcal S_{\overline\Omega}(x)$ (resp. $\mathcal S^{\mathrm r}_{\overline\Omega}(x)$) the set of curves $\gamma$ in $\overline\Omega$ starting at $x$ and solving the differential inclusion $\gamma'(t)\in \{f(t,\gamma(t),u):u\in U\}$ a.e.~$t\in[0,T]$ (resp. $\gamma'(t)\in \operatorname{conv}\{f(t,\gamma(t),u):u\in U\}$); see Section \ref{sec:FW} for details. 
Our result gives a counterexample to the Filippov-Wazewski relaxation Theorem with state constraint (see \cite[Corollary 3.2]{frankowska2000filippov}) in the case where one of its assumptions, namely the Soner-type inward pointing condition, is violated along a curve of the boundary of the state constraint.
It thus shows that there is no hope for the classical Filippov--Wa\v zewsky result to hold in the presence of state constraints without additional hypotheses.

\begin{theorem}\label{thm:FW-intro}
For $d\geq 5$, there exist a bounded open set $\domain \subset\R^d$ with smooth boundary, a non-convex compact set of controls $U\subset\R^2$, a smooth controlled vector field $\f$, a point $\initial \in\domain$, positive numbers $\delta, \varepsilon>0$, and a curve $\gamma$ in $\mathcal S^{\mathrm r}_{\overline\Omega}(\initial)\setminus \mathcal S_{\overline\Omega}(\initial)$ such that, for all $x\in\R^d$ with $\|x-\initial\|\leq \delta$ and all $\theta\in  \mathcal S_{\overline\Omega}(x)$,
\[\|\gamma-\theta\|_\infty\geq\varepsilon.\]
\end{theorem}

This will be proved and further discussed in Section \ref{sec:FW}.

\subsection{Ideas and organization of the paper}
\label{sec:mechanism}

\paragraph{Idea of the mechanism.} 
In order to obtain a gap, we look for a situation in which the minimizer cannot be approximated well enough by nearby admissible curves. To achieve this in the sub-Riemannian setting, our strategy is to start with a minimizer $\eta$ that simply follows one of the vector fields generating the sub-Riemannian structure, namely, $\eta'(t)=f_1(\eta(t))=(1,0,\dots,0)$; $\eta$ is thus just a straight line. Then construct $\domain$ to be diffeomorphic to a ball while spiraling around $\eta$ in a way that the sub-Riemannian velocity distribution at $\eta$ remains tangent to the boundary of the domain $\domain$ (see Figure \ref{fig:Omega}). 
This way, we cut away part of the space that would be necessary for admissible curves to approximate $\eta$ both in position and velocity (i.e., with similar controls and similar cost), and instead the spiral interacts with the sub-Riemannian structure to force admissible curves to become more oscillatory the closer they get to $\eta$ and their cost increases.
Also, while $\eta$ is an admissible curve, since the running cost has four wells in the control variable located at the points $(\pm1,\pm1)$, the cost of $\eta$ as a curve  with control $u=(1,0)$ is larger than the cost of the Young measure $\nu_t=\tfrac12(\delta_{(1,1)}+\delta_{(1,-1)})$ supported at two controls $(1,\pm1)$ inducing speeds $f_1\pm f_2$ whose average is $\eta'=f_1$.

\paragraph{Overview of the proofs and organization of the paper.} 
To prove Theorem \ref{thm:example}, we first describe the geometric objects involved in Section \ref{sec:setting}. Then, in Section \ref{sec:proofs} we prove an important auxiliary result, Theorem \ref{thm:gaps}, that is a version of Theorem \ref{thm:example} but with a non-smooth running cost. This is eventually replaced by a smooth running cost in Section \ref{sec:regularityL}, where the proof of Theorem \ref{thm:example} (restated as Theorem \ref{thm:smoothlagrangian}) is finished. 

Theorem \ref{thm:example2} is deduced from Theorem \ref{thm:smoothlagrangian} (restated as Theorem \ref{thm:secondexample}) in Section \ref{sec:vinterlike} using a well-known extension technique that we review in Section \ref{sec:nullLagrangian} that turns any Bolza problem into a Mayer problem. We also use the equivalence (reviewed in Section \ref{sec:convexificaton}) between the Young measure relaxation and the relaxation consisting in the convexification of the differential inclusion.

In Section \ref{sec:targetset} we deduce Theorems \ref{thm:example4} and \ref{thm:example3} (restated as Theorem \ref{gap_withoutX}); here the trick consists of leveraging a terminal cost that, upon optimizing, enforces the condition that the minimizers end in $X$. %

The failure of the Filippov-Wa\v zewski approximation is discussed in Section \ref{sec:FW}, where Theorem \ref{thm:FW-intro} is rephrased as Corollary \ref{coro:noFW}.

\paragraph{Acknowledgements.}
We are very grateful to Piernicola Bettiol, Fr\'ed\'eric Jean, Dario Prandi, Mario Sigalotti, and Ugo Boscain for enlightening discussions.

This work was co-funded by the European Union under the
project Robotics and advanced industrial production (ROBOPROX, reg.~no.~CZ.02.01.01/00/22\_008/0004590).
This work was partially supported by UiT Aurora Center for Mathematical
Structures in Computations MASCOT. It was also partially supported by the project Pure Mathematics in Norway funded by Trond Mohn Foundation and Tromsø Research Foundation. It was also partially supported by ANR--3IA Artificial and Natural Intelligence Toulouse Institute [ANR--19--PI3A--0004].
This research was also part of the programme DesCartes and is supported by the National Research Foundation, Prime Minister's Office, Singapore under its Campus for Research Excellence and Technological Enterprise (CREATE) programme.

\section{General setting}
\label{sec:setting}
Let $d\geq 4$ be the dimension of the state space.
Throughout the paper, we will use four parameters $\delta>0$, $0<a<1<b$, $1<\lambda<2$, with $ab>2\pi$ that will be further specified in Assumptions \ref{asm:deltaGamma}, \ref{ass:nonemptycontenders}, and \ref{asm:ballboxonOmega}.

In the introduction, we stated our theorems claiming that the domain $\domain$ was a ball. From this point on, however, we will use the same notation $\domain$ for a spiraling set that is only \textit{diffeomorphic} to a ball. We choose this setting because it is easier to visualize the mechanism, do the relevant calculations, and articulate the proofs; the reader should keep in mind that we are simply working up to diffeomorphism (indeed, all the objects in the theorems will have to be transformed by the diffeomorphism), and there should be no risk of confusion.

The setting described in this section corresponds to Theorems \ref{thm:example} and \ref{thm:example4}; the setting for Theorems \ref{thm:example2} and \ref{thm:example3} are closely related, passing through the conversion procedure described in Sections \ref{sec:vinterlike} and \ref{sec:nullLagrangian}.

\subsection[A Goursat-type sub-Riemannian structure]{A Goursat-type sub-Riemannian structure on {$\R^d$}}

Consider the $d$-dimensional Euclidean space $\R^d$ endowed with a rank 2 sub-Riemannian structure, defined by the vector fields
\begin{align*}
    f_1(x)&=(1,0,0,\dots,0)\in\R^d,\\
    f_2(x)&=(0,1,x_1,\tfrac12 x_1^2,\tfrac1{3!} x_1^3,\dots,\tfrac{1}{(d-2)!}x_1^{d-2}),
    \end{align*}
    so that the only nonzero brackets of $f_1$ and $f_2$ are
    \begin{align*}
        f_3(x)&\coloneqq[f_1,f_2](x)=(0,0,1,x_1,\tfrac12 x_1^2,\tfrac1{3!} x_1^3,\dots,\tfrac{1}{(d-3)!}x_1^{d-3}),\\
        f_4(x)&\coloneqq[f_1,f_3](x)=(0,0,0,1,x_1,\tfrac12 x_1^2,\tfrac1{3!} x_1^3,\dots,\tfrac{1}{(d-4)!}x_1^{d-4}),\\
        &\cdots\\\
        f_d(x)&\coloneqq[f_1,f_{d-1}](x)=(0,0,0,0,0,0,0,1).
    \end{align*}
    This is a sub-Riemannian structure of Goursat type with rank 2. 
    We associate to it the controlled vector field
    \begin{equation}\label{eq:def-srf}
        f(x,u)=u_1f_1(x)+u_2f_2(x).
    \end{equation}

\begin{definition}\label{def:horizontal}
    [$U$-horizontal curves and $U$-admissible Young measures]    
    Given a set $U\subset\R^2$, an absolutely continuous curve  $\gamma\colon I\to\R^d$ defined on an interval $I\subseteq\R$ is $U$-horizontal if
    there is a measurable function $u\colon I\to U$ such that $\gamma'(t) =f(\gamma(t),u(t))$ for almost every $t\in I$. 
    Similarly, we will say that the Young measure $(\nu_t)_t$ is \emph{$U$-admissible} if, for almost every $t$, the support of $\nu_t$ is contained in $U$.

    The sub-Riemannian arc length $\ell(\gamma)$ of a horizontal curve $\gamma$ with control $u$ is 
    \[\ell(\gamma)=\int_I\sqrt{u_1(t)^2+u_2(t)^2}dt.\]

\end{definition}

    The vector fields $\{f_1,f_2,\dots,f_d\}$ generate $\R^d$ at each point. %
    The Chow-Rashevskii theorem \cite[Th.~3.31]{ABB20} implies that, for any $U$ linearly spanning all of $\R^2$ (e.g., if $U$ contains a neighborhood of the origin), %
    the sub-Riemannian distance $\dist_U(x,y)$ between points $x,y\in\R^d$,
    \begin{equation*} 
        \dist_U(x,y) =\inf\left\{\ell(\gamma)\,\middle|\,
         \gamma\colon[0,T]\to\R^d\text{ $U$-horizontal},\;
        \gamma(0)=x,\;\gamma(T)=y,\;T>0\right\},
    \end{equation*}
    is finite and continuous. In fact, that theorem also implies that we can join every two points in $\R^d$ with a horizontal curve $\gamma$. %

\subsection[Definition of the domain]{Definition of the domain $\domain$}
\label{sec:domain}

    Let $\mathfrak p\colon[-a,a]\to[0,1]$ be a $C^\infty$ function equal to 1 on $[-a+\tfrac{1}{b^2},a-\tfrac{1}{b^2}]$ and equal to 0 for $t<-a+\tfrac{1}{b^3}$ and for $t>a-\tfrac{1}{b^3}$.
    We will use two curves, $\eta\colon\R\to\R^d$ and $\xi\colon\R\to\R^{d-1}$  given  by\footnote{In a first reading, the reader may assume that $\mathfrak p(t)=1$; this will only affect the smooth diffeomorphism result of Proposition \ref{prop:gammadelta}.}
    \begin{align}
        \label{eq:def-eta} \eta(t)&=(t,\underbrace{0,0,\dots,0}_{d-1})\in\R^d,\\
         \xi_b(t)&=\mathfrak p(t)(\underbrace{0,0,\dots,0}_{d-3},\tfrac1b\cos(bt),\tfrac1b\sin(bt)) %
        +(1-\mathfrak p(t))(\underbrace{0,0,\dots,0}_{d-3},\tfrac1b\cos(ab),\tfrac1b\sin(ab))\in\R^{d-1}.
        \label{eq:def-xi}
    \end{align}
    Observe that the curve $\xi_b$, when projected onto the plane of the last two coordinates $x_{d-1},x_d$, rotates around the origin $ab/2\pi$ times, and is constant in intervals $(-\infty,-a+\tfrac{1}{b^3}]$ and $[a-\tfrac{1}{b^3},\infty)$.
    
    Let 
    \[\ad=a+\delta,\]
    and, recalling $1<\lambda<2$,
    \[\mathbf x_0=\eta(-\lambda a),\qquad\mathbf x_1=\eta(
    \lambda a).\]
    We also let $A_t$ be a $d\times d$ matrix whose columns are $f_1\circ\eta(t),f_2\circ\eta(t),\dots,f_d\circ\eta(t)$; for example, for $d=4$, %
    it looks like this:
    \[
    A_t=(f_1\circ\eta(t),f_2\circ\eta(t),\dots,f_4\circ\eta(t))\\
    =\begin{pmatrix}
        1&0&0&0\\
        0&1&0&0\\
        0&t&1&0\\  
        0&\tfrac12t^2&t&1\\      
    \end{pmatrix}\in \mathrm{GL}_4(\R),\qquad t\in\R.
    \]
    Note that $A_t^{-1}=A_{-t}$.

    It will be useful to define the map $\varphi\colon\R^d\to\R^d$ given by
    \begin{equation}\label{eq:defvarphi}
        \textstyle\varphi(x)=A_{x_1}x
        =(x_1,x_2,x_1x_2+x_3,\tfrac12x_1^2x_2+x_1x_3+x_4,\dots,\sum_{j=2}^{d}\tfrac{1}{(d-j)!}x_1^{d-j}x_j),
    \end{equation}
    for any $x = (x_1,\ldots,x_d)\in \R^d$. Its inverse is
    \[\varphi^{-1}(x)=A_{-x_1}x,\qquad x\in\R^d.\]
    Note however that $\varphi$ is not a linear map. The role of the map $\varphi$ will be to straighten-up the coordinates of the boxes involved in the Ball-Box Theorem; cf.~Corollary \ref{coro:ballbox}.

    Let also $\pi\colon\R^d\to\R^{d-2}$ be the projection onto the plane $x_1=0=x_2$, that is, 
    \begin{equation}\label{eq:defpi}
        \pi(x)=(x_3,x_4,\dots,x_d),\qquad x\in\R^d.
    \end{equation}
    We define
    \begin{equation}
        \label{eq:defr}
        r(x)=\|\pi\circ\varphi^{-1}(x)\|=\sqrt{\varphi^{-1}(x)_3^2+\varphi^{-1}(x)_4^2+\dots+\varphi^{-1}(x)_d^2}, \qquad x\in\R^d.
    \end{equation}
    
    Let 
    \begin{align*}
        \mathbb S
    & = \textstyle\varphi\left(\bigcup_{x_1\in[-a,a]}\{x_1\}\times B^{d-1}(\xi_b(x_1),1/b)\right)\\
    &=\{\varphi(x):x\in\R^d,\; -a\leq x_1\leq a,\;\\
    &\qquad\qquad\qquad x_2^2+\dots+x^2_{d-2}+%
    (x_{d-1}-\tfrac1b\cos(bx_1))^2+(x_d-\tfrac1b\sin(bx_1))^2\leq 1/b^{2}\},
    \end{align*}
    where $B^{d-1}(p,r)\subset\R^{d-1}$ is the $(d-1)$-dimensional ball centered at $p$ of radius $r$, taken with respect to the Euclidean norm.

    \begin{proposition} \label{prop:gammadelta}
        There are $a,\delta>0$, $1<\lambda<2$, and an open set $\Gamma\subset(\R\setminus[-a,a])\times\R^{d-1}\subseteq\R^d$ such that the following are true:
        \begin{enumerate}[label=\roman*.,ref=(\roman*)]
         \item The set
    \begin{align}\label{eq:defOmega}
        \domain\coloneqq \mathbb S\cup \Gamma
    \end{align} 
    is diffeomorphic to the ball $B^d(0,1)$, contains $\mathbf x_0$ and $\mathbf x_1$, and its closure contains the image $\eta([-\ad,\ad])\subseteq\overline\domain$; refer to Figure \ref{fig:Omega}.
         \item Given a point $p$ in $\varphi(\{-a\}\times \overline B^{d-1}(\xi_b(-a),1/b))$ (respectively, a point $q$ in $\varphi(\{a\}\times \overline B^{d-1}(\xi_b(a),1/b))$) there is a horizontal curve $\gamma$ with controls in $\cornersU\coloneqq\{(\pm1,\pm1)\}$ of length at most $ \delta/2$ joining the endpoint $\mathbf x_0$ with $p$ (respectively, joining $q$ with $\mathbf x_1$), whose image is completely contained in $\Gamma\cup \varphi(\{a\}\times \overline B^{d-1}(\xi_b( a),1/b))\cup \varphi(\{-a\}\times \overline B^{d-1}(\xi_b(-a),1/b))$; in other words, the curve must stay in $\Gamma$ except perhaps at $p$ (resp. $q$), where it may touch the boundary of $\Gamma$.
    \end{enumerate}
    \end{proposition}
    \begin{assumpt}\label{asm:deltaGamma}
We will assume that $a$, $\lambda$, $\Gamma$, and $\delta$ are as in Proposition \ref{prop:gammadelta}.
\end{assumpt}
    \begin{proof}[Proof of Proposition \ref{prop:gammadelta}]        
        Let $P\colon[a,(2\lambda-1) a]\to [0,2/b]$ be a smooth function with $P(a)=1/b$, $P'(a)=0$, $P'(t)>0$ for $a<t<\lambda a$, $P'(t)<0$ for $\lambda a<t<(2\lambda-1)a$, $P((2\lambda-1)a)=0$, and $\lim_{t\nearrow (2\lambda-1)a}P'(t)=-\infty$, and such that the inverse of the restriction $(P|_{(\lambda a,(2\lambda-1)a)})^{-1}$ is smooth; refer to Figure \ref{fig:P}.
        
        Take $\Gamma$ to be the set
        \begin{align*}
            \Gamma=&\{\varphi(t,x): a\leq t\leq (2\lambda-1)a,\; x\in  B^{d-1}(\xi_b(t),P(t))\}\\
            &\cup \{\varphi(t,x):  -(2\lambda-1)a\leq t\leq -a,\; x\in  B^{d-1}(\xi_b(t), {P(-t)})\}.
        \end{align*}
        By the properties above, $\Gamma\cup \mathbb S$ is a set with smooth boundary that contains $\mathbf x_0$ and $\mathbf x_1$. Since the boundary of $\Gamma\cup \mathbb S$ is smooth , and from the construction it is clear that this set is homeomorphic to the ball $B^d(0,1)$; by \cite[Thm.~C]{palais}, these two sets are also diffeomorphic. 
        Moreover, $\Gamma$ contains connected open sets $D_+$ and $D_-$ that, respectively, contain the sets $S_\pm=\varphi(\{\pm \lambda a\}\times \overline B^{d-1}(\xi(\pm\lambda a),1/b)) $. It follows from Chow's theorem that there are $\cornersU$-horizontal curves within $D_-$ (within $D_+$) joining $\mathbf x_0$ ($\mathbf x_1$) with every point of $S_-$ ($S_+$), in an amount of time that can be uniformly bounded from above, say by $M>0$. Also, if $1<\lambda<2$ is taken close to 1 (i.e., if we take $\lambda-1$ small), then the image 
        \[
            \varphi\left(\bigcup_{t\in[-\lambda a,-a]\cup[a,\lambda a]}\{t\}\times \overline B^{d-1}\left(\xi_b(\operatorname{sgn}(t) \lambda a),\frac{1}{b}\right)\right)
        \]
        of the pair of cylinders 
        is contained in $\Gamma$, so
        every point in $S_\pm$ can be joined to a point in $\mathbb S$ in time $\leq 2(\lambda-1)a$, by alternating controls $(1,1)$ and $(1,-1)$ in $\cornersU$ to go in the approximate direction of $f_1$; this can be easily formalized using \cite[Th.~2.3]{frankowska2000filippov}. This means that we can take any $\delta\geq 2(M+ 2(\lambda-1)a)$ to fulfill the statement of the proposition.
    \end{proof}
    \begin{figure}[t]
\centering
\includegraphics[width=7cm%
]{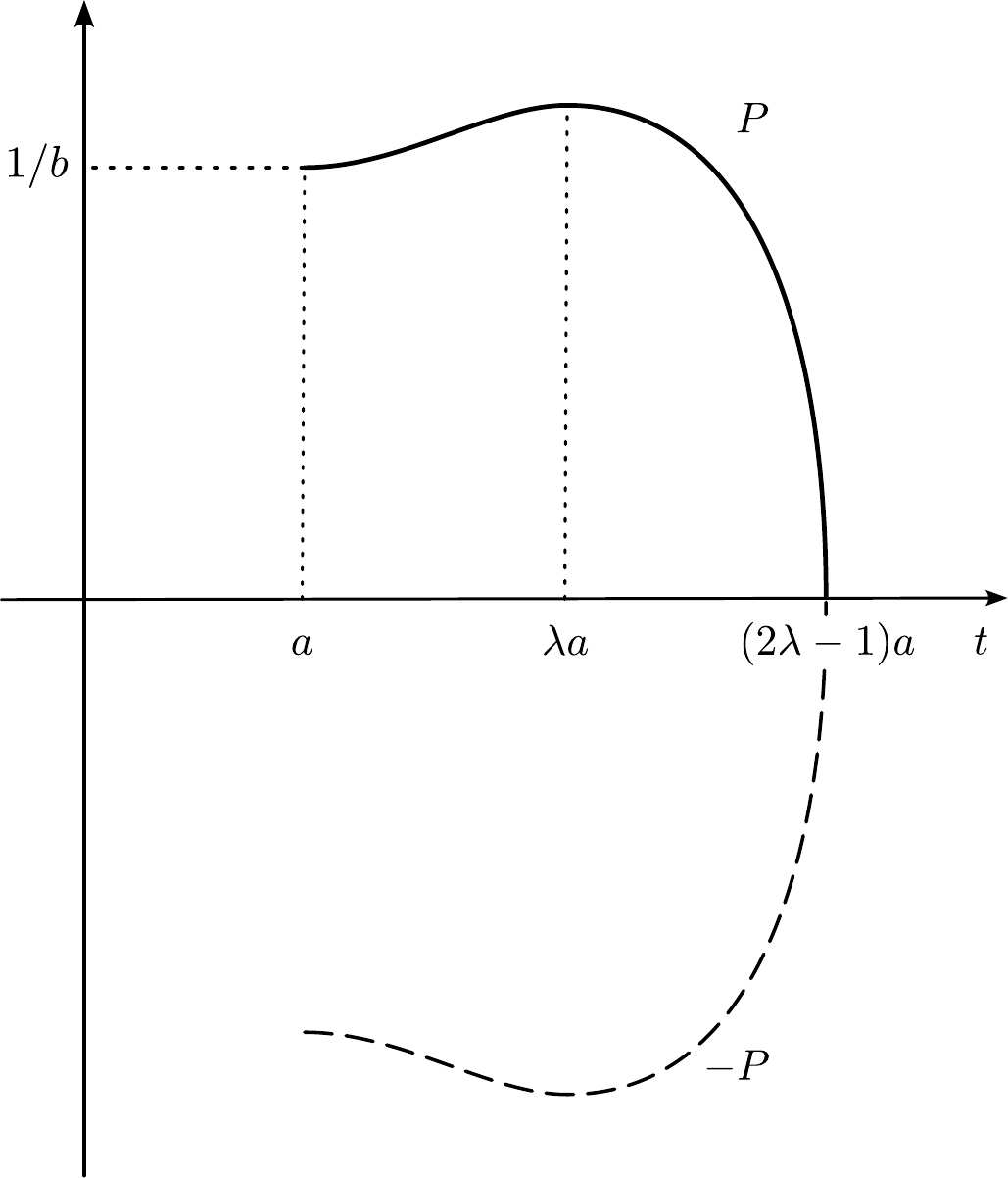}
\caption{The function $P$  in the proof of Proposition \ref{prop:gammadelta}, and its negative $-P$, which together delineate the smooth profile of $\Gamma$.}
\label{fig:P}
\end{figure}

    Observe that, for $x\in\overline{\mathbb S}$, having $r(x)=0$ implies also $x_2=0$ (but not viceversa). Also, focusing on the function $\eta$ defined in \eqref{eq:def-eta}, we note that the image of $\eta|_{[-a,a]}$ is contained in $\overline{\mathbb S}\subset\overline\domain$.

    \begin{figure}[t]
\centering
\includegraphics[width=\textwidth%
]{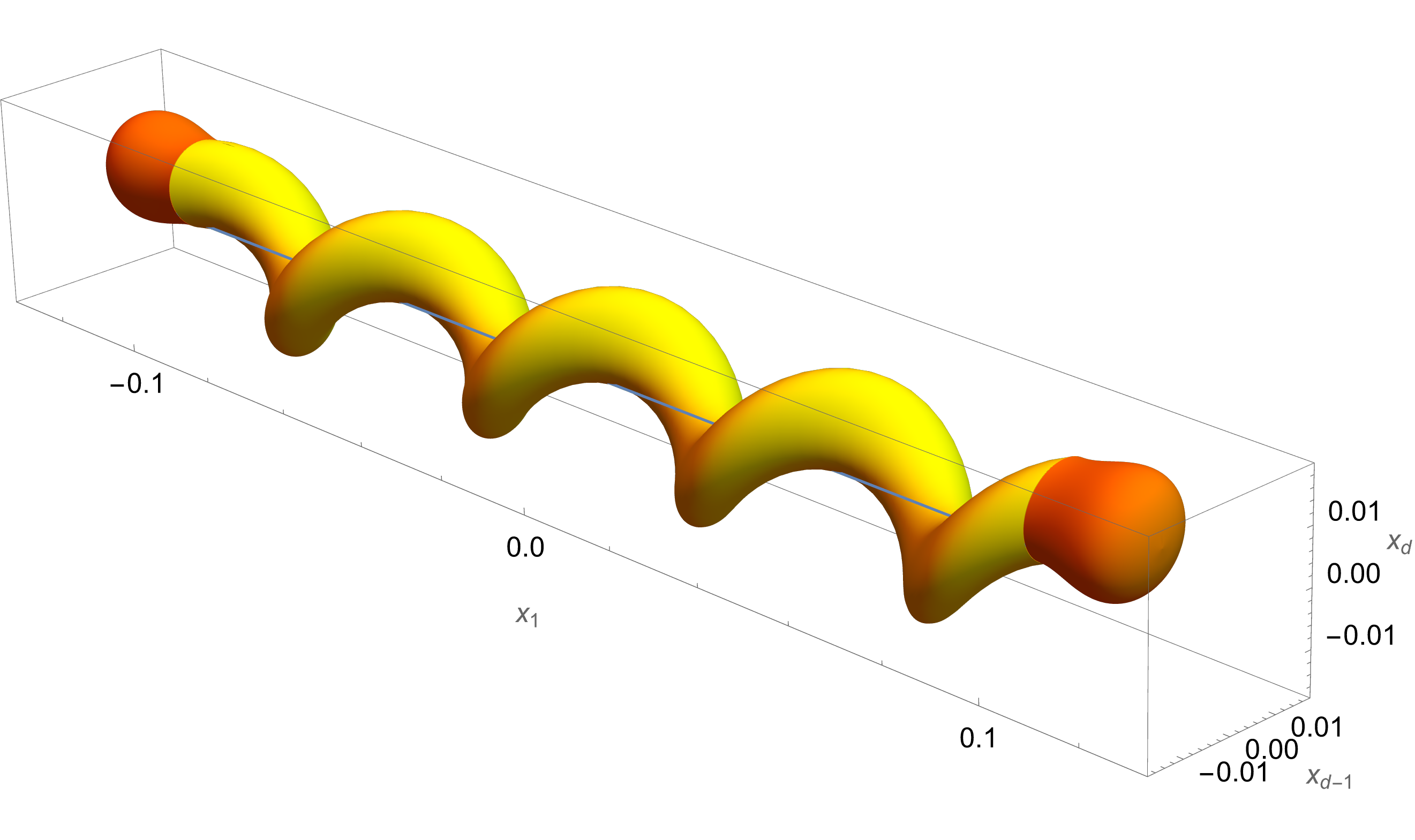}
\caption{Image of the domain $\domain\subseteq\R^d$ under the projection $x\mapsto(x_1,x_{d-1},x_d)\in\R^3$. For this illustration, we set $a=1/10$ and $b=40\pi$. The set $\mathbb S$ is portrayed in yellow, and $\Gamma$, in orange. Although the image seems to show a singularity in the boundary between the orange and yellow parts, this is actually smooth because $\mathfrak p$ is $C^\infty$. We have also drawn the image $\eta([-\ad,\ad])$ in blue. %
} 
\label{fig:Omega}
\end{figure}

\subsection{The role of initial and target sets.}
\label{sec:role_in_tar}

For simplicity, in our discussion we will take the initial and target sets to be the single points $\initial$ and $\mathbf x_1$, so that 
\[\destination=\{\mathbf x_1\}.\]
Remark however that the initial and target sets can be replaced by any subsets of $\Gamma\cap\{x_1<-a-\epsilon\}\subset\domain$ and $\Gamma\cap\{x_1>a+\epsilon\}\subset\domain$, respectively, %
as long as they contain $\initial$ and $\mathbf x_1$ and $\epsilon>0$ is small (as determined by the proof of Theorem \ref{thm:smoothlagrangian}), without undermining the results and
with very little modification to the arguments we give in the proofs.
This is because the Lagrangian densities involved in the proofs, namely, $\nonconvexL$, $\convexL$, and $\bar L_\epsilon$  (to be defined below) all vanish on $\Gamma\cap \{|x_1|>a+\epsilon\}$, so what happens there is mostly immaterial to the optimization problems, and the gap is only caused by the requirement to turn along the spiral defined by $\Omega$.

Observe that, taking a diffeomorphism of the ball $B\subset\R^d$ into $\domain$ (as was shown to exist in Proposition \ref{prop:gammadelta}) that takes $V\cap \overline B$ into $\Gamma\cap\{x_1>a+\epsilon\}$, guarantees that property \ref{it:targetconstraint} will be satisfied.

\subsection{Lagrangian densities and sets of controls}\label{sec:lagrangians_controls}

     Let 
    \[\U=[-1,1] \times[-1,1] \] 
    be the square centered at the origin, and notice that the notation is motivated by the shape of the set. 
    In the following we will also fix a closed, convex set $U$ satisfying 
    \[\U\subseteq U\subseteq \R^2.\]
    The set $U$ may be unbounded.%

    Let also, for $x\in\R^d$ with $-a\leq x_1\leq a$ and for $u=(u_1,u_2)\in\R^2$,
    \begin{equation}\label{eq:deflagrangians}
    \begin{aligned}
        \convexL(x,u)&=\chi_{[-a,a]}(x_1)\left(\|u\|_\infty%
        +r(x)^2\right),\\
        \nonconvexL(x,u)&=\chi_{[-a,a]}(x_1)\left(1+2\Delta(u)+r(x)^2\right),
    \end{aligned}
    \end{equation}
    where  $r(x)$ is defined in~\eqref{eq:defr} and
    \begin{equation}\label{eq:asterisk}
    \begin{aligned}
        \|u\|_\infty&=\max(|u_1|,|u_2|),\\
        \Delta(u)&=\dist(u,\{(\pm1,\pm1)\})=\min_{i,j\in\{-1,+1\}}\dist(u,(i,j)),\\
        \chi_{[-a,a]}(x_1)&=\begin{cases}
            1,&-a\leq x_1\leq a,\\
            0,&\text{otherwise.}
        \end{cases}
    \end{aligned}
    \end{equation}
    The function $\Delta$ is the distance to the corners of the square $\U$, so it has four minima or `wells.'
    Observe that 
    $\convexL(x,u)\leq \nonconvexL(x,u)$,
    and that $\convexL$ %
    is convex in $u$, while $\nonconvexL$ is not. 
    The notation is motivated by the shape of the graph of each of the functions, for fixed $x$, as seen from the side: %
    $\convexL(x,\cdot)$ has a single minimum and its graph looks like a cone, outlining a $\vee$; and $\nonconvexL(x,\cdot)$ has four wells, and seen from the side we can see two of them, so the lower contour of the graph looks like $\vee\!\vee$.

 The proof of Theorem \ref{thm:example} will involve a smooth approximation to $\nonconvexL$, as we will explain in  Section \ref{sec:regularityL}. We first need a preliminary result, Theorem \ref{thm:gaps}, involving $\nonconvexL$ directly; we describe this next. The reason to define $\convexL$ and $\U$ is that the proof of Theorem \ref{thm:gaps} heavily relies on  comparisons between problems involving $\nonconvexL$ and $U$ and analogous problems involving $\convexL$ and $\U$.

\section{Preliminary auxiliary results}
\label{sec:proofs}
\subsection{Optimization problems compared}    \label{MAIN_RESULT}

    We will state an important assumption and the main result of this section. The definitions of the objects involved in Theorem \ref{thm:gaps}, as well as a discussion of its significance, and its proof, can be found below. The notations change slightly from those of \eqref{min:orig}--\eqref{min:relaxed}, %
    so we make all definitions explicit.

    Since $\domain$ is a connected open set and the sub-Riemannian structure is equiregular, any two points in it can be joined with a horizontal curve \cite{jean2003entropy}, so we may make the following assumption.
    \begin{assumpt}\label{ass:nonemptycontenders}
        We assume that $\delta>0$ is large enough such 
        that there exists an absolutely continuous curve $\gamma\colon[-\ad,\ad]\to \domain$, $\gamma(-\ad)=\initial$, $\gamma(\ad)\in \destination$, with $\gamma'(t)\in f(t,\gamma(t),\U)$ for almost every $t\in[0,T]$.
    \end{assumpt}   
    \begin{remark}[Regarding reachability]
        Assumption \ref{ass:nonemptycontenders} is equivalent to saying that the set of contenders of problem \eqref{min:6} below is nonempty.
        Since a contender for problem \eqref{min:6} will also be represented among the contenders of problems \eqref{min:1}, \eqref{min:2}, and \eqref{min:5} below,  all the quantities $\mincurves(\nonconvexL,U)$,  $\minarcs(\nonconvexL,U)$, $\mathring \mincurves(\convexL,\U)$, $\mincurves(\convexL,\U)$ defined by those problems are finite. The reader will find that Assumption \ref{ass:nonemptycontenders} has the same effect for the problems of Section \ref{sec:regularityL} involving the Lagrangians $\bar L_\epsilon$ and $L$. This also means that Assumption \ref{ass:nonemptycontenders} does away with all reachability issues for all those problems.
    \end{remark}

  In Theorem \ref{thm:gaps} we compare the following optimization problems: 
    \begin{itemize}
        \item Non-convex Lagrangian and convex set of controls, optimization over horizontal curves: 
            \begin{customopti}%
                {inf}{u}
                {\int_{-\ad}^{\ad} \nonconvexL(\gamma(t),u(t))\,dt}{\label{min:1}}{%
                \mincurves(\nonconvexL,U)=}
                \addConstraint{u\colon[-\ad,\ad]\to U\;\text{measurable}}
                \addConstraint{\gamma'(t)=f(t,\gamma(t),u(t))\text{ a.e.~$t$}}
                \addConstraint{\gamma(t)\in\overline\domain\;\text{for all $t$}}
                \addConstraint{\gamma(-\ad)=\initial,\;\gamma(\ad)\in\destination}
            \end{customopti}
        \item Non-convex Lagrangian and convex set of controls, optimization over Young measures:
            \begin{customopti}%
                {inf}{(\nu_t)_t}
                {\int_{-\ad}^{\ad}\int_{U} \nonconvexL(\gamma(t),u)\,d\nu_t(u)\,dt}{\label{min:2}}{%
                \minarcs(\nonconvexL,U)=}
                \addConstraint{\text{$(\nu_t)_t$ is a Young measure on $U$}}
                \addConstraint{\text{$\gamma'(t)=\int_Uf(t,\gamma(t),u)\,d\nu_t(u)$ a.e.~$t$}}
                \addConstraint{\gamma(t)\in\overline\domain\;\text{for all $t$}}
                \addConstraint{\gamma(-\ad)=\initial,\;\gamma(\ad)\in\destination}
            \end{customopti}

   \end{itemize}

\begin{theorem}\label{thm:gaps}
    In the situation defined in the previous section \eqref{eq:def-srf} and \eqref{eq:defOmega}, and with Assumptions \ref{asm:deltaGamma}, \ref{ass:nonemptycontenders}, and \ref{asm:ballboxonOmega}, we have
    \begin{equation}
        \label{eq:minmain} \mincurves(\nonconvexL,U)>\minarcs(\nonconvexL,U)
        \qquad\text{i.e., \eqref{min:1}}>\text{\eqref{min:2}},
    \end{equation}
\end{theorem}

  For the proof of the theorem, we will need to consider these two auxiliary problems  (with \eqref{min:6} being less standard):
    \begin{itemize}
        \item Convex Lagrangian, compact and convex set of controls, optimization over horizontal curves $\gamma$ that are almost never at $r(\gamma)=0$ (equivalently, $\|\pi(\gamma(t))\|>0$ a.e. $t\in[-\ad,\ad]$): %
              \begin{customopti}
                {inf}
                {u}
                {\int_{-\ad}^{\ad} \convexL(\gamma(t),u(t))\,dt}{\label{min:6}}{%
                \mathring\mincurves(\convexL,\U)=}
                \addConstraint{u\colon[-\ad,\ad]\to\U\;\text{measurable}}
                \addConstraint{\gamma'(t)=f(t,\gamma(t),u(t))\;\text{a.e.~$t$}}
                \addConstraint{\gamma(t)\in\overline\domain\;\text{for all $t$}}
                \addConstraint{\gamma(-\ad)=\initial,\;\gamma(\ad)\in\destination}
                \addConstraint{r(\gamma(t))\neq 0 \;\text{for almost every $t\in[-\ad,\ad]$.}}
            \end{customopti}
        \item Convex Lagrangian, compact and convex set of controls, optimization over horizontal curves (possibly staying for a while in $r(\gamma)=0$):  
              \begin{customopti}
                {inf}
                {u}
                {\int_{-\ad}^{\ad} \convexL(\gamma(t),u(t))\,dt}{\label{min:5}}{%
                \mincurves(\convexL,\U)=}
                \addConstraint{u\colon[-\ad,\ad]\to\U\;\text{measurable}}
                \addConstraint{\gamma'(t)=f(t,\gamma(t),u(t))\;\text{a.e.~$t$}}
                \addConstraint{\gamma(t)\in\overline\domain\;\text{for all $t$}}
                \addConstraint{\gamma(-\ad)=\initial,\;\gamma(\ad)\in\destination}
            \end{customopti}
    \end{itemize}

    \begin{proof}[Proof of Theorem \ref{thm:gaps}]

    From the lemmas in Section \ref{sec:curveequalities}, there are $0<a<1<b$ 
    such that the following are true:
    \begin{gather}
        \label{eq:curveequalities} \mincurves(\convexL,\U)= 2a,
        \qquad \minarcs(\nonconvexL,U)= 2a,
    \end{gather}
    and, under the (ultimately false) assumption that $\mincurves(\nonconvexL,U)=\minarcs(\nonconvexL,U)$, we also have
    \begin{equation}\label{eq:conditionalone}
        \mathring \mincurves(\convexL,\U)\leq \mincurves(\nonconvexL,U).
    \end{equation}
    In Section \ref{sec:mininterior} we will show that, if $d\geq 4$,
    \begin{equation}
        \label{eq:hardone}
         2a%
         <\mathring \mincurves(\convexL,\U).
        \end{equation}
    From these, it follows immediately that
        \begin{equation}
        \label{eq:mininterior} \mathring \mincurves(\convexL,\U)>\mincurves(\convexL,\U) 
        \qquad \text{i.e., \eqref{min:6}}>\text{\eqref{min:5}}. 
    \end{equation}
     To obtain \eqref{eq:minmain} one may reason as follows: either \eqref{eq:minmain} is true, or if not, then \eqref{eq:conditionalone} holds, and one obtains a contradiction using \eqref{eq:curveequalities} and \eqref{eq:hardone}.
     \end{proof}

\begin{remark}[Gap with non convex controls]
     The result of Theorem \ref{thm:gaps} can be directly adapted to show a gap in a Lagrange problem with a nonconvex control set, which is not the main focus of this paper but is also interesting.
     Consider the set of controls
     \[\cornersU=\{(\pm1,\pm1)\}\subset \R^2\]
     with the running cost given by 
     \[\constantL(x,u)=\chi_{[-a,a]}(x_1)(1+r(x)^2),\]
     again with notations inspired by the shape of the set $\cornersU$, which consists of the corners of a square, and by the look of the graph of $\constantL(x,\cdot)$ from the side, which is simply a horizontal plane.
     These naturally induce the following two problems:
         \begin{itemize}
                     \item Non-convex set of controls, optimization over horizontal curves:            
                \begin{customopti}%
                    {inf}{u}
                    {\int_{-\ad}^{\ad} \constantL(\gamma(t),u(t))\,dt}{\label{min:3}}{%
                    \mincurves(\constantL,\cornersU)=}
                    \addConstraint{\text{$u\colon[-\ad,\ad]\to \cornersU$ measurable}}
                    \addConstraint{\gamma'(t)=f(t,\gamma(t),u(t))\;\text{a.e.~$t$}}
                    \addConstraint{\gamma(t)\in\overline\domain\;\text{for all $t$}}
                    \addConstraint{\gamma(-\ad)=\initial,\;\gamma(\ad)\in\destination}
                \end{customopti}
            \item
                Non-convex set of controls, optimization over Young measures:
                \begin{customopti}
                    {inf}{(\nu_t)_t}
                    {\int_{-\ad}^{\ad}\int_{\cornersU} \constantL(\gamma(t),u)\,d\nu_t(u)\,dt}{\label{min:4}}{%
                    \minarcs(\constantL,\cornersU)=}
                    \addConstraint{\text{$(\nu_t)_t$ a Young measure on $\cornersU$}}
                    \addConstraint{\gamma'(t)=\int_{\cornersU}f(t,\gamma(t),u\,d\nu_t(u)\;\text{a.e.~$t$}}
                    \addConstraint{\gamma(t)\in\overline\domain\;\text{for all $t$}}
                    \addConstraint{\gamma(-\ad)=\initial,\;\gamma(\ad)\in\destination}
                \end{customopti}
        \end{itemize}
        To guarantee that the sets of contenders are nonempty, and hence that these numbers are finite, it may be necessary to take $\delta>0$ larger than as stated in Assumption \ref{ass:nonemptycontenders}; let us assume that this is the case.
        Then we have the following gap:
                \begin{equation}\label{eq:minU} 
            \mincurves(\constantL,\cornersU)>\minarcs(\constantL,\cornersU)
            \qquad\text{i.e., \eqref{min:3}}>\text{\eqref{min:4}}.
        \end{equation}
        Let us show why this is true.
        First, we have that $\minarcs(\constantL,\cornersU)\leq 2a$
        by the same argument as in Lemma \ref{lem:2}. From Lemma \ref{lem:8} with $Y=\constantL$, it follows that $\minarcs(\constantL,\cornersU)\geq 2a$, so in fact $\minarcs(\constantL,\cornersU)= 2a$. 
        We also have 
        \[\mathring \mincurves(\convexL,\U)\leq \mincurves(\constantL,\cornersU)\]
        because
        the set of contenders for \eqref{min:3} is contained in the set of contenders for \eqref{min:6}; this follows from the facts that $\cornersU\subset \U$ and that an absolutely continuous curve $\gamma$ satisfying $r(\gamma(t))=0$ for $t$ in some subset of positive measure cannot be a contender in \eqref{min:3}: such a curve $\gamma$ must satisfy $\gamma'(t)=c(t)f_1(t)$ in a set of positive measure, for some measurable function $c$; this means that its control must be of the form $(c(t),0)\notin \cornersU$, contradicting the conditions in \eqref{min:3}.
        The gap \eqref{eq:minU} then follows from \eqref{eq:curveequalities} and \eqref{eq:mininterior}.
\end{remark}

\subsection{Results needed for the proof of Theorem \ref{thm:gaps}}
\label{sec:proof}

Section \ref{sec:curveequalities} collects the technical lemmas needed to deduce equations \eqref{eq:curveequalities} and \eqref{eq:conditionalone}, which are significantly easier than \eqref{eq:hardone}, to which the rest of the section is devoted.

The setting we describe in Section \ref{sec:setting} features a domain $\domain$ that spirals around the straight curve $\eta$ that turns out to be the integral curve of the relaxed minimizer. 
The sub-Riemannian structure makes it very hard to traverse the spiral through any curve other than $\eta$. %
A curve $\gamma$ that almost never coincides with $\eta$ must go around the spiral, and in Sections \ref{s:ball} and \ref{sec:costofturning} we show that this must make its cost significantly larger than that of $\eta$, even as $\eta$ and $\gamma$ may be arbitrarily close to each other. 

In Section \ref{sec:costofturning}, most of the work is devoted to ensuring that we can apply a result in topology known as the Dog-on-a-Leash Theorem \cite[Th.~3.11]{fulton2013algebraic}, which we use to formalize the intuitive notion that, if a curve $\gamma$ traverses a domain that spirals $k$ times around $\eta$, then the curve necessarily rotates at least $k$ times around $\eta$. Once we have established that fact, we also obtain some estimates for the cost by applying the Uniform Ball-Box Theorem \cite{jean2001uniform}. Qualitatively, the state constraint $\Omega$ forces the approximating trajectories $\gamma$ to achieve rotations in the $(x_{d-1},x_d)$-plane  along the spiral, and provides a fixed lower bound for the corresponding winding number given in Lemma~\ref{lem:mustturn}. Since the latter plane is spanned by the directions given by Lie brackets of highest orders (also known as the nonholonomic directions), it turns out that the cost of turning becomes too expensive in order to have a suitable approximation of the curve $\eta$.

In Section \ref{sec:aprioribounds} we prove some a priori estimates that show that any curve that is almost a minimizer must be very close to $\eta$. We conclude the proof of \eqref{eq:hardone} in Section \ref{sec:mininterior}.

\subsubsection[Proof of two equations]{Proof of %
equations \texorpdfstring{\eqref{eq:curveequalities}}{} %
and \eqref{eq:conditionalone}{}}
\label{sec:curveequalities}

The equations in \eqref{eq:curveequalities} will be settled in Lemmas \ref{lem:1}, \ref{lem:2}, and \ref{lem:8}. Inequality \eqref{eq:conditionalone} will be dealt with in Lemma \ref{lem:4}.

\begin{lemma}\label{lem:1}
    $\mincurves(\convexL,\U)\leq 2a$.
\end{lemma}
\begin{proof}
    Observe that the curve 
    \[\psi(t)= \begin{cases}
        \mathbf x_0,&-\ad\leq t\leq -\lambda a,\\
        \eta(t),&-\lambda a\leq t\leq \lambda a,\\
        \mathbf x_1,&\lambda a\leq t\leq \ad,
    \end{cases}\]
    where the curve $\eta$ is as defined in \eqref{eq:def-eta}, together  with the control $u(t)=(1,0)$ for $t\in[-\lambda a, \lambda a]$ and $u(t)=(0,0)$ for $t\in[-a_\delta,a_\delta]\setminus[-\lambda a, \lambda a]$, %
    is a valid contender in \eqref{min:5}, and the corresponding integral equals $2a$. This shows that $\mincurves(\convexL,\U)\leq 2a$.
\end{proof}

\begin{lemma}\label{lem:2}
    $\minarcs(\nonconvexL,U)\leq 2a$.
\end{lemma}
\begin{proof}
    Take the curve $\psi$ from the proof of Lemma \ref{lem:1}, together with the measures
    \[\nu_t=\begin{cases}
        \tfrac14\sum_{p\in\{(\pm1,\pm1)\}}\delta_p,& -\ad\leq t\leq -\lambda a\text{  or  }\lambda a\leq t\leq \ad,\\
        \tfrac12(\delta_{(1,1)}+\delta_{(1,-1)}),&-\lambda a<t<\lambda a.
    \end{cases}
    \]
    Then again $(\psi,\nu_t)$ is a valid contender in \eqref{min:2} and in \eqref{min:4}, and the corresponding integral can be checked to equal $2a$.
\end{proof}
\begin{lemma}\label{lem:8}
    If the integrand $Y\colon \overline\domain\times U\to\R$ satisifies $Y(x,u)\geq u_1$, and if $(\gamma,\nu_t)$ is an $U$-admissible Young measure, such that $\gamma(-\ad)=\mathbf x_0$, $\gamma(\ad)=\mathbf x_1$, then
    \[\int_{-\ad}^{\ad}\int_{U}Y(\gamma,u)\,d\nu_t(u)\,dt\geq 2a.\]
    As a consequence, we have $\mincurves(\convexL,\mathbf U)\geq 2a$ and $\minarcs(\nonconvexL,U)\geq 2a$.
\end{lemma}
\begin{proof}
    Let $X\subseteq[-\ad,\ad]$ be the (closed) set of points $t$ with $\gamma_1\in[-a,a]$. We have 
    \begin{multline*}
        \int_{-\ad}^{\ad}\int_{U}Y(\gamma,u)\,d\nu_t(u)\,dt
        =\int_{X}\int_{U}Y(\gamma,u)\,d\nu_t(u)\,dt
        \geq \int_{X} \int_{U}u_1\,d\nu_t(u)\,dt\\
        =\sum_{[c,d]\subseteq X}\int_c^d\int_{U}u_1\,d\nu_t(u)\,dt
        =\sum_{[c,d]\subseteq X}\int_c^d\gamma'_1(t)\,dt 
        =\sum_{[c,d]\subseteq X}(\gamma_1(d)-\gamma_1(c))=2a.
    \end{multline*}
    Here, the sums are taken over all maximal intervals $[c,d]$ contained in $X$. 
    
    The last assertion of the lemma follows by letting $Y$ be $\convexL$, $\constantL$, and $\nonconvexL$, respectively, and noticing that contenders  in \eqref{min:4}  are a subset of the contenders in \eqref{min:2}, and the contenders in \eqref{min:5} are canonically represented by a subset of the contenders in \eqref{min:2}.
\end{proof}

\begin{lemma}\label{lem:4}
    If $\mincurves(\nonconvexL,U)=\minarcs(\nonconvexL,U)$, then
    $\mathring \mincurves(\convexL,\U)\leq \mincurves(\nonconvexL,U)$.
\end{lemma}

Recall the well-known \emph{Markov inequality}: for any Borel measure $\mu$ on a set $X\subset\R^d$, and every measurable function $f\colon X\to\R$, 
\begin{equation}\label{eq:markov}
    \mu(\{x\in X:|f(x)|\geq \varepsilon\})\leq \frac1\varepsilon\int_{X}|f(x)|\,d\mu(x).
\end{equation}

\begin{proof}[Proof of Lemma \ref{lem:4}]%
    Assume that $\mincurves(\nonconvexL,U)=\minarcs(\nonconvexL,U)$.
    By Lemma \ref{lem:2}, we have the bound $\minarcs(\nonconvexL,U)\leq 2a$.
    Let $0<\varepsilon<\delta$, and let $A_\varepsilon$ be the set of contenders $(\gamma,u)$ in \eqref{min:1} satisfying 
    \begin{equation}\label{eq:smallaction}
      \int_{-\ad}^{\ad}\nonconvexL(\gamma,u)dt
      \leq \mincurves(\nonconvexL,U)+\varepsilon
      =\minarcs(\nonconvexL,U)+\varepsilon
      \leq 2a+\varepsilon.
    \end{equation}
    Denote by $B_\varepsilon$ the set of contenders in the following problem:
        \begin{customopti}%
            {inf}{\tilde\gamma,\tilde u}
            {\int_{-\ad}^{\ad} \convexL(\tilde\gamma(t),\tilde u(t))\,dt}{\label{min:varepsilon}}{%
            M_\varepsilon=}
            \addConstraint{\tilde\gamma\colon[-\ad,\ad]\to\overline\domain\;\,\text{a $\U$-horizontal curve}}
            \addConstraint{\text{$\tilde u$ is the control of $\tilde\gamma$}}
            \addConstraint{\tilde\gamma(-\ad)=\mathbf x_0,\;\tilde\gamma(\ad)=\mathbf x_1}
            \addConstraint{|\{t\in[-\ad,\ad]:r(\tilde\gamma(t))= 0\}|\leq \sqrt\varepsilon.}
        \end{customopti}
    We will now show that, to each contender $(\gamma,u)\in A_\varepsilon$, corresponds a contender $(\xi,u_\xi)\in B_\varepsilon$ with 
        \begin{equation}\label{eq:wantedineq}
            M_\varepsilon\leq\int_{-\ad}^{\ad} \convexL(\xi, u_{\xi})dt \leq \int_{-\ad}^{\ad} \nonconvexL(\gamma,u)dt.
        \end{equation}
    The statement of the lemma will follow from 
        \[  \mincurves(\nonconvexL,U)
            =\liminf_{\varepsilon\searrow0} \inf_{(\gamma,u)\in A_\varepsilon}\int \nonconvexL(\gamma,u)dt 
            \geq \liminf_{\varepsilon\searrow0} M_\varepsilon 
            =\mathring \mincurves(\convexL,\U),
        \]
    as the contenders of \eqref{min:6} are $\bigcap_\varepsilon B_\varepsilon$.
    
    As an intermediate step, we will now prove that, for $(\gamma,u)\in A_\varepsilon$, the amount of time  spent by $\gamma$ in the set $\overline \domain\cap r^{-1}(0)$ is 
    \begin{equation}\label{eq:Esmall}
        |E|\leq \sqrt\varepsilon,\qquad E=\{t\in[-\ad,\ad]:r(\gamma(t))=0\}.
    \end{equation}
    To show this, observe that the function
    \[\phi(u)\coloneqq 1+2\Delta(u)-\|u\|_\infty,\]
    with $\Delta$ as in \eqref{eq:asterisk},
    is nonnegative $\phi\geq0$; it also satisfies, for $x\in\R^d$ with $x_1\in[-a,a]$,
    \[\phi(u)=\chi_{[-a,a]}(x_1)(1+2\Delta(u)-\|u\|_\infty)=\nonconvexL(x,u)-\convexL(x,u);\]
    and its sublevel sets $\phi^{-1}((-\infty,\alpha]))$ for $0<\alpha\leq 1$ have four connected components, each of them a compact neighborhood of a point of $\cornersU$, of diameter asymptotically vanishing as $\alpha\searrow 0$, and homeomorphic to a closed disc.
    Apply the Markov inequality \eqref{eq:markov}  to see that the Lebesgue measure $|P|$ of the set
    \[P=\{t\in[-\ad,\ad]: -a\leq\gamma_1(t)\leq a, \;\phi(u(t))\geq \sqrt\varepsilon%
    \}\]
    satisfies
    \begin{equation*}
        |P|
        \leq \frac{1}{\sqrt\varepsilon}\int_{-\ad}^{\ad}\chi_{[-a,a]}(\gamma_1(t))\phi(u(t))dt
        =\frac{1}{\sqrt\varepsilon}\int_{-\ad}^{\ad} \nonconvexL(\gamma,u)-\convexL(\gamma,u)\,dt
        \leq \frac{1}{\sqrt\varepsilon}(2a+\varepsilon-2a)=\sqrt\varepsilon,
    \end{equation*}
    by \eqref{eq:smallaction} and since, by  Lemma \ref{lem:8}, %
    \[\int_{-\ad}^{\ad}\convexL(\gamma,u)\,dt \geq 2a.\]
    By continuity of $\gamma$, the set $E$ is a union of closed intervals.
    If $I\subset E$ is an interval of positive length, then $\gamma'(t)$ is a multiple of $e_1$ for $t\in I$, and $u(t)$ is a multiple of $(1,0)$. This means that for $t\in I$, we must have $\phi(u(t))>\sqrt\varepsilon$, i.e., $t\in P$. This in turn means that $E\subseteq P$. As a consequence, we have $|E|\leq|P|\leq \sqrt\varepsilon$, so we have proved \eqref{eq:Esmall}.

    Let $(\gamma,u)\in A_\varepsilon$ with $0<\varepsilon<\delta$. 
    Let $(\tilde\gamma,\tilde u)$ be a pair verifying the following: 
    \begin{itemize}
        \item $\tilde\gamma$ and $\tilde u$ are defined on an interval $I\subset\R$,
        \item the curve $\tilde\gamma$ is $\U$-horizontal with control $\tilde u$,
        \item there is a reparameterization $\sigma\colon[-\ad,\ad]\to I$ such that $\tilde\gamma(\sigma(t))=\gamma(t)$, 
        \item $\tilde u(\sigma(t))=u(t)/\|u(t)\|_\infty$ for almost every $t$ with $u(t)\neq 0$, so that also $\|\tilde u(s)\|_\infty=1$ for almost every $s\in I$, and
        \item $\sigma(u^{-1}(0))$ is a set of measure zero.
    \end{itemize}
    Observe, in particular, that $\sigma$ is an increasing function, and is hence differentiable almost everywhere by Lebesgue's theorem.
    From the above properties we have, for almost every $t\in [-\ad,\ad]$,
    \begin{multline*}
        u_1(t)f_1(\gamma(t))+u_2(t)f_2(\gamma(t))
        =\gamma'(t)%
        =\frac{d}{dt}\tilde\gamma(\sigma(t))%
        =\tilde\gamma'(\sigma(t))\sigma'(t))\\
        =(\tilde u_1\circ\sigma(t)f_1(\tilde\gamma\circ\sigma(t))+\tilde u_2\circ \sigma(t)f_2(\tilde\gamma\circ\sigma(t)))\sigma'(t)%
        =\frac{u_1(t)f_1(\gamma(t))+ u_2(t)f_2(\gamma(t)))}{\|u(t)\|_\infty}\sigma'(t).
    \end{multline*}
    Thus we have, for almost every $t$,
    \[\sigma'(t)=\|u(t)\|_\infty.\]
    Let $Y\subseteq I$ be the set of $t\in I$ with $\tilde\gamma_1(t)\in[-a,a]$. By continuity of $\tilde\gamma$, $Y$ is a closed set, and since $\tilde\gamma$ joins $\mathbf x_0$ with $\mathbf x_1$ and travels within $\overline\domain$, there must be an interval $[t_0,t_1]\subset Y$ with $\tilde\gamma_1(t_0)=-a$, $\tilde\gamma_1(t_1)=a$. We must have $t_1-t_0\leq 2a+\delta$, for otherwise we would have
    \begin{multline*}
        2a+\delta>2a+\varepsilon\geq\int_{\sigma^{-1}([t_0,t_1])}\nonconvexL(\gamma,u)\,dt
        \geq\int_{\sigma^{-1}([t_0,t_1])}\convexL(\gamma,u)\,dt\\
        \geq\int_{\sigma^{-1}([t_0,t_1])}\|u(t)\|_\infty dt
        =\int_{\sigma^{-1}(t_0)}^{\sigma^{-1}(t_1)}%
        \sigma'(t)\,dt
        =\sigma\circ\sigma^{-1}(t_1)-\sigma\circ\sigma^{-1}(t_0)=
        t_1-t_0\geq 2a+\delta.
    \end{multline*}
    Thus we can take the restriction $\tilde\gamma|_{[t_0,t_1]}$ and concatenate it to two curves joining, respectively, $\mathbf x_0$ to $\tilde\gamma(t_0)$ and $\tilde\gamma(t_1)$ to $\mathbf x_1$, each of them of length $\leq\delta/2$ (see Assumption \ref{asm:deltaGamma}); extend the resulting curve to a $\U$-horizontal curve $\xi$ defined in the interval $[-\ad,\ad]$, whose length is $2\ad=2a+2\delta$. Now, $\xi$ and its control $u_\xi$ form a pair with $|\{t\in[-\ad,\ad]:r(\xi(t))= 0\}|\leq \sqrt\varepsilon$ (by our argument above; cf.~\eqref{eq:Esmall}), so that $(\xi,u_\xi)\in B_\varepsilon$ satisfying \eqref{eq:wantedineq}. This concludes the proof of the lemma.
\end{proof}

\subsubsection{The ball-box theorem and some consequences}\label{s:ball}

\paragraph{The uniform ball-box theorem.}
    We state the version of the theorem that we will use, which is a simplified version of what the reader will find in \cite{jean2001uniform}.

    Let $\R^d$ be endowed with a sub-Riemannian structure induced by analytic vector fields $X_1,\dots,X_m$ with commutators $[X_i,X_j]\coloneqq X_iX_j-X_jX_i$.  
    
    For a multiindex $I=(i_1,\dots,i_k)$ with entries $1\leq i_j\leq m$, we set $|I|=i_1+i_2+\dots+i_k$. We also let
    \[[X_I]=[X_{i_1},[X_{i_2},\dots [X_{i_{k-1}},X_{i_k}]]\dots],\]
    be the corresponding commutator of order $k$; when $k=1$, this means that $[X_{(i)}]=X_i$ for $i=1,\dots,m$. We assume that the vector fields $\{[X_I]:|I|\leq n\}$ generate the entire %
    space $\R^d$ %
    at every point
    (i.e., they satisfy the Chow condition). 
    
    For each $1\leq k\leq n$, let $\mathcal L^k(p)\subset \R^d$ %
    be the subspace spanned by $\{[X_I]:|I|\leq k\}$. %
    Assume that all points $p\in M$ are regular in the sense that the vector spaces $\mathcal L^k(p)$ have constant dimension $\dim \mathcal L^k(p)$ globally for all $p$; we will thus use the shorthand $\dim \mathcal L^k$ instead of $\dim \mathcal L^k(p)$. 

    The family  $\{[X_{I_1}],\dots,[X_{I_n}]\}$ of commutators is said to be a \emph{minimal basis} if, at each $p$, the vectors $[X_{I_1}](p),\dots,[X_{I_n}](p)$ generate $\R^d$, and if $|I_j|=s$ whenever $\dim \mathcal L^{s-1}<j\leq \dim \mathcal L^{s}$ (setting $\dim \mathcal L^0=0$).

    With a fixed minimal basis $\{[X_{I_1}],\dots,[X_{I_n}]\}$, let $\operatorname{Box}(p,\rho)$ the set of all points $q$ that can be reached from $p$ by 
    \[q=\phi^{[X_{I_n}]}_{r_n}\cdots\phi^{[X_{I_2}]}_{r_2}\phi^{[X_{I_1}]}_{r_1}(p), \qquad |r_k|\leq \rho^{|I_k|},\]
    where $\phi^X_t(P)$ denotes the flow of the vector field $X$ for time $t$ starting at the point $P$, i.e., $\phi^X_0(P)=P$ and $d\phi^X_t(P)/dt=X(\phi^X_t(P))$.
    
    Denote by $\overline B^{SR}(p,\rho)$ the closed sub-Riemannian ball centered at $p$ of radius $\rho>0$, consisting of all points $q$ that can be joined to $p$ with a horizontal curve of arc length $\leq \rho$. 
    
\begin{theorem}[Uniform ball-box {\cite[Cor.~3]{jean2001uniform}}, simplified]\label{thm:ballbox}
    Fix a minimal basis $\{[X_{I_1}],\dots,[X_{I_n}]\}$.
    Let $K\subset \R^d$ be a compact set. There exist constants $c,C,\delta_0>0$ such that, for every $p\in K$, $0<\varepsilon\leq \delta_0$, we have
    \[\operatorname{Box}(p,c\varepsilon)\subseteq \overline B^{SR}(p,\varepsilon)\subseteq\operatorname{Box}(p,C\varepsilon).\]
\end{theorem}

Let us now specialize this theorem to our situation. Let $M=\R^d$, let $K$ be a compact set, $n=d$, $m=2$, and
\[I_1=(1),\qquad I_2=(2),\qquad I_k=(1,2,\underbrace{1,1,\dots,1}_{k-3}),\qquad  \text{for $3\leq k\leq d$},\]
so that $[X_{I_k}]=f_k$ for $1\leq k\leq d$.
This gives:

\begin{corollary}\label{coro:ballbox}
    There exist constants $\ballboxsmall,\ballbox,\delta_0>0$ such that, for every $p\in K\supseteq \overline\domain$ and $0<\rho\leq \delta_0$,
    \begin{multline*}
        \varphi(y+[-\ballboxsmall\rho,\ballboxsmall\rho]\times [-\ballboxsmall\rho,\ballboxsmall\rho]\times[-\ballboxsmall\rho^2,\ballboxsmall\rho^2]\times[-\ballboxsmall\rho^3,\ballboxsmall\rho^3]\times\cdots\times [\ballboxsmall\rho^{d-1},\ballboxsmall\rho^{d-1}])\\
        \subseteq
        \overline B^{SR}(p,\rho)
        \subseteq\\
        \varphi(y+[-\ballbox\rho,\ballbox\rho]\times [-\ballbox\rho,\ballbox\rho]\times[-\ballbox\rho^2,\ballbox\rho^2]\times[-\ballbox\rho^3,\ballbox\rho^3]\times\cdots\times [\ballbox\rho^{d-1},\ballbox\rho^{d-1}]),
    \end{multline*} 
    with $\varphi$ as defined in \eqref{eq:defvarphi} and $\varphi(y)=p$. 
\end{corollary}
\begin{proof}
    Immediately from the definition we see that $[X_{I_k}]=f_k$ is a minimal basis. Apply Theorem \ref{thm:ballbox}. Denote the standard basis of $\R^d$ by $e_1,e_2,\dots,e_n$. Observe that the boxes $\operatorname{Box}(p,c\rho)$ and $\operatorname{Box}(p,C\rho)$ are precisely of the form given in the statement of the corollary because $\phi^{f_1}_r(p)=p+re_1=\varphi(\phi^{e_1}_r(y))$ (because of the way $\varphi$ is defined \eqref{eq:defvarphi}), and $\phi_r^{f_k}(\varphi(x))=\varphi(\phi_r^{e_k}(x))=\varphi(x+re_k)$ for $x\in \R^d$, $r\in\R$ and $2\leq k\leq n$, since $D\varphi(q) e_k=f_k(\varphi(q))$ for $q\in\R^d$. We can take $\ballboxsmall=c$ and $\ballbox=C$ to obtain the statement of the corollary.
\end{proof}
\begin{assumpt}\label{asm:ballboxonOmega}
    Observe that we can take $K$ to be a large box centered at the origin, and by taking  $0<a<1<b$ small enough and large enough, respectively, we may assume that $\overline{\mathbb S}\subset K$ and $2/b<\delta_0$.
\end{assumpt}
\paragraph{Ball-box on $\tilde\gamma$.}Throughout the rest of this section we fix a contender $(\gamma,u)$ in \eqref{min:6}. We will denote 
\[\tilde\gamma(t)=\varphi^{-1}(\gamma(t)),\qquad t\in[-\ad,\ad]\]
Then  $\gamma_1(t)=\tilde\gamma_1(t)$, and it follows from the definition \eqref{eq:defr} of $r$ that
\[r(\gamma(t))=\sqrt{\tilde\gamma_3(t)^2+\tilde\gamma_4(t)^2+\dots+\tilde\gamma_d(t)^2}.\]

 We will now get some estimates on certain polygonal approximations to $\tilde\gamma$. 

\begin{lemma}\label{lem:ballboxapplied}
    Let $\delta_0>0$ be as in Corollary \ref{coro:ballbox}.
    There are constants $C>0$ and $0<s_0\leq \delta_0$ such that, if  
    $0\leq s\leq s_0$, then
    \[\left|\tilde\gamma_j(t+s)-\tilde \gamma_j(t)\right|\leq Cs_0^{j-1} %
    ,\qquad j=3,\dots,d.
    \]
        In particular, if $i\in\mathbb N $ is such that $1/i^{5/2}<s_0$,  then
    \begin{equation}\label{eq:estimateballboxnew}
        \left|\tilde\gamma_j(t+s)-\tilde\gamma_j(t)\right|\leq Ci^{\frac52(2-d)}, \qquad0\leq s\leq 1/i^{5/2},\; j=d-1,d.
    \end{equation}
\end{lemma}
\begin{proof}
    Let $0\leq s\leq s_0\leq \delta_0$ %
    and $3\leq j\leq d$. %
    Note that, since $\gamma$ is $\mathbf U$-horizontal and $\|u\|\leq \sqrt2$ for all $u\in\mathbf U$, we have that  $\gamma(t+s)$ is in the sub-Riemannian ball $\overline B^{SR}(\gamma(t),\sqrt 2s)$.
    Thus if we take $0\leq s\leq s_0\leq \delta_0/\sqrt 2$, $\gamma(t+s)\in \overline B^{SR}(\gamma(t),\delta_0)$, and (also by Assumption \ref{asm:ballboxonOmega}) we may apply Corollary \ref{coro:ballbox} with $\rho=s_0$, $p=\gamma(t)$, and $y=\tilde\gamma(t)$; the corollary gives
    \[\tilde\gamma_j(t+s)\in y_j+ [-\bar C\rho^{j-1},\bar C\rho^{j-1}]=\tilde\gamma_j(t)+[-\bar Cs_0^{j-1},\bar Cs_0^{j-1}].\]
    In other words, 
    \[|\tilde\gamma_j(t+s)-\tilde\gamma_j(t)|\leq \bar Cs_0^{j-1},\]
    so we can take $C=\bar C$.
    The particular case \eqref{eq:estimateballboxnew} follows by taking $s_0=1/i^{5/2}$ and noticing that $i^{\frac52}(1-d)<i^{\frac52}(2-d)$.
\end{proof}

\subsubsection{The cost of turning}
\label{sec:costofturning}

In this section we will establish some estimates necessary for the proof of \eqref{eq:hardone}.

Let $\tau\colon\R^d\to\R^2$ be the projection 
\[\tau(x)=(x_{d-1},x_{d}),\qquad x\in\R^d.\] 

Let $(\gamma,u)$ be a contender in \eqref{min:6}. We will use the same notations as in the previous section. Let also
\[\mathring\gamma(t)=\tau(\varphi^{-1}(\gamma(t))+(0,\xi_b(\gamma_1(t)))),\qquad t\in[-\ad,\ad].\]
Let us briefly discuss the idea behind this definition.
Recall that $\varphi^{-1}$ can be thought of as straightening the sub-Riemannian structure up, and $\varphi^{-1}\circ\gamma$ is the curve corresponding to $\gamma$ in that straightened-up version of things. We need a curve that is not only straightened-up, but also pushed away from $\eta$; this is what the term $(0,\xi_b(\gamma_1(t))$ is intended to do. The curve $\mathring \gamma$ is the projection onto the two-dimensional $(x_{d-1},x_d)$-plane of a slight translation of $\varphi^{-1}\circ\gamma$ away from the origin. In this section we aim to prove estimates stemming from the fact that $\gamma$ turns around $\eta$ roughly as many times as $\domain$ spirals. The curve $\mathring\gamma$ is an auxiliary curve that will be useful at showing this, as it turns around $\eta$ about as many times as $\gamma$, but importantly $\mathring \gamma$ never touches the origin, which allows us to apply the Dog-on-a-Leash Theorem to a curve obtained by closing it up to form a loop.

Consider the vector field
\[\Theta(x_{d-1},x_d)=\frac{(-x_d,x_{d-1})}{x_{d-1}^2+x_d^2},\qquad (x_{d-1},x_d)\neq 0.\]
If we let $\theta$ be the angle, in the $x_{d-1}x_d$-plane, with the $x_{d-1}$ axis, namely,
\[\theta(x_{d-1},x_d)=\arctan \frac{x_d}{x_{d-1}},\qquad x_{d-1}\neq0,\]
then $\Theta=\nabla \theta$, when the latter is well defined.
It is standard (see for example \cite[Ch.~2]{fulton2013algebraic}) that, if the curve $\alpha\colon [0,T]\to\R^2\setminus \{0\}$ is absolutely continuous, there is an integer $k\in\mathbb Z$ such that
\[ \int_{0}^T\langle\Theta\circ\alpha(t),\alpha'(t)\rangle\,dt=\theta(\alpha(T))-\theta(\alpha(0))+2\pi k\]
is the angle traversed by $\alpha$ as it goes around the origin $0\in\R^2$. If $\alpha(0)=\alpha(T)$, $k$ is the number of (oriented) turns $\alpha$ completes around the origin.

\begin{lemma}\label{lem:mustturn}
    Every interval $I=[i_0,i_1]\subseteq[-\ad,\ad]$ such that $\gamma(i_0)=-a\leq \gamma_1(t)\leq a=\gamma(i_1)$ for all $t\in I$ verifies %
    \[
        \int_I \langle\Theta\circ\mathring\gamma(t),\mathring\gamma'(t)\rangle dt 
        \geq ab-2\pi.
    \]
\end{lemma}
\begin{proof}
    Recall that $ab>2\pi$.    
    Let $I=[i_0,i_1]\subseteq[-\ad,\ad]$ be an interval as in the statement of the lemma. %
    Let $\alpha,\beta\colon[0,1]\to B^2(0,2/b)\setminus\{0\}\subset\R^2$ be smooth curves with $\alpha(0)=\tau\circ\mathring\gamma(i_1)$, $\alpha(1)=\tau\circ\mathring\gamma(i_0)$, $\beta(0)=2\tau(0,\xi_b(a))$, and $\beta(1)=2\tau(0,\xi_b(-a))$, chosen so that 
    \begin{equation}\label{eq:dogleash}
        \left|\int_0^1 \langle\Theta\circ\alpha,\alpha'\rangle dt\right|\leq\pi,\qquad \left|\int_0^1\langle\Theta\circ\beta,\beta'\rangle dt\right|\leq\pi,
    \end{equation}
    and
    \[\|\alpha(t)-\beta(t)\|\leq 1/b=\|\beta(t)\|.\]
    Let $\nu_1\colon[i_0,i_1+1]\to B^2(0,3/b)\setminus\{0\}$ (respectively, $\nu_2\colon[i_0,i_1+1]\to B^2(0,3/b)\setminus\{0\}$) be the concatenation of $\tau\circ\tilde\gamma|_I$ and $\alpha$ (of $2\tau(0,\xi_b\circ\gamma_1)|_I$ and $\beta$). The curves $\nu_1$ and $\nu_2$ are closed, and satisfy $\|\nu_1-\nu_2\|\leq 1/b\leq \|\nu_2\|$, so by the Dog-on-a-Leash Theorem \cite[Th.~3.11]{fulton2013algebraic} we have
    \[\int_{i_0}^{i_1}\langle\Theta\circ\nu_1,\nu_1'\rangle dt=\int_{i_0}^{i_1}\langle\Theta\circ\nu_2,\nu_2'\rangle dt.\]
    The statement of the lemma follows from this together with properties \eqref{eq:dogleash}.
\end{proof}

For $j=1,2,\dots$, let $V_j\subseteq[-\ad,\ad]$ be defined by
\[V_j=\{t\in [-\ad,\ad]:\gamma_1(t)\in[-a,a],\;r(\gamma(t))\in(\tfrac1{j+1},\tfrac1j]\}.\]
The sets $V_j$ are measurable, and $R_\gamma\setminus\bigcup_jV_j$ is a set of measure zero; cf.~\eqref{min:6}.

\begin{lemma}\label{lem:partitionestimatesnew}
Let $I$ be an interval as in Lemma \ref{lem:mustturn}. 
    For each $j\in\mathbb N$, take numbers $ t_{j,0}<t_{j,1}<\dots<t_{j,N_j}\in I\cap (V_{j-1}\cup V_j\cup V_{j+1})$, such that, if we let
    \[\ell_j\coloneqq\frac1{j^{5/2}},\]
    then
    the intervals $[t_{j,k},t_{j,k}+\ell_j)$ are disjoint and cover $I\cap V_j$. 
    Then:
    \begin{enumerate}[label=\roman*.,ref=(\roman*)]
            \item \label{it:arclengthnew} For $j=1,2,\dots$ such that $j^{-5/2}<s_0$, with $s_0$ as in Lemma \ref{lem:ballboxapplied}, %
            the  polygonal approximation to $\tau\tilde\gamma$ obtained by linearly interpolating the points $(\gamma_{d-1}(t_{j,k}),\gamma_d(t_{j,k}))$ and $(\gamma_{d-1}(t_{j,k}+\ell_j),\gamma_d(t_{j,k}+\ell_j))$ in $V_j$ has arc length 
            \begin{equation*}%
                \sum_{k=0}^{N_j}\sqrt{(\tilde\gamma_{d-1}(t_{j,k}+\ell_j)-\tilde\gamma_{d-1}(t_{j,k}))^2+(\tilde\gamma_d(t_{j,k}+\ell_j)-\tilde\gamma_d(t_{j,k}))^2} %
                \leq C%
                \frac{|V_{j-1}|+|V_j|+|V_{j+1}|}{j^{\frac52(d-3)}}.
            \end{equation*}
        \item \label{it:angleboundnew}
            We can estimate the winding of $\mathring\gamma|_I$ around the origin by
            \begin{equation*}
              \left|\int_{I}\langle \Theta\circ\mathring\gamma,\mathring\gamma'\rangle dt\right| %
              \leq\sum_{j=1}^\infty 2j
            \sum_{k=0}^{N_j}\sqrt{(\tilde\gamma_{d-1}(t_{j,k}+\ell_j)-\tilde\gamma_{d-1}(t_{j,k}))^2+(\tilde\gamma_d(t_{j,k}+\ell_j)-\tilde\gamma_d(t_{j,k}))^2}. %
            \end{equation*}
    \end{enumerate}
\end{lemma}
\begin{proof}
    To obtain estimate \ref{it:arclengthnew}, apply \eqref{eq:estimateballboxnew} on the intervals $[t_k,t_{k+1}]$ covering $V_i$; an upper bound of the number of intervals necessary is, by construction of the sets $V_i$, $(|V_{i-1}|+|V_i|+|V_{i+1}|)/\ell_i$. 

    Let us now prove estimate \ref{it:angleboundnew}. Let $\zeta\colon I\to\R^2$ be the polygonal approximation  obtained by linearly interpolating the points $\tau\circ\tilde\gamma(t_{j,k})=(\tilde\gamma_{d-1}(t_{j,k}),\tilde\gamma_d(t_{j,k}))$ and $\tau\circ\tilde\gamma(t_{j,k}+\ell_j)$. From the definition of $V_j$, we know that the distance from $\|\zeta(t_{j,k}+s)\|\in[\tfrac{1}{i+1},\tfrac{1}{i-1}]$ for $0\leq s\leq \ell_j$. It follows from \eqref{eq:estimateballboxnew} that the segment joining the points $\zeta(t_{j,k})=\tau\circ\tilde\gamma(t_{j,k})$ and $\zeta(t_{j,k}+\ell_j)=\tau\circ \tilde\gamma(t_{j,k}+\ell_j)$ does not intersect the origin 0, so we can estimate the angle traversed by $\zeta$ by
    \begin{multline}\label{eq:stepone}
        \left|\int_{t_k}^{t_{k+1}}\langle\Theta\circ\zeta,\zeta'\rangle dt\right|
        \leq \frac{\text{distance travelled}}{\text{radius}}\leq \frac{\|\zeta(t_{j,k}+\ell_j)-\zeta(t_{j,k})\|}{\tfrac{1}{j+1}}\\
        \leq 2j\|\zeta(t_{j,k}+\ell_j)-\zeta(t_{j,k})\|
        =2j\|\tau\circ\tilde\gamma(t_{j,k}+\ell_j)-\tau\circ\tilde\gamma(t_{j,k})\|,
    \end{multline}
    as per the definition of radians. 
    Adding over all the relevant $k$ and $j$ gives 
     \begin{equation*}
              \left|\int_{\bigcup_jV_j
              }\langle \Theta\circ\zeta,\zeta'\rangle dt\right|%
              \leq\sum_{j=1}^\infty 2j
           \sum_{k=0}^{N_j}\sqrt{(\tilde\gamma_{d-1}(t_{j,k}+\ell_j)-\tilde\gamma_{d-1}(t_{j,k}))^2+(\tilde\gamma_d(t_{j,k}+\ell_j)-\tilde\gamma_d(t_{j,k}))^2}
            \end{equation*}

    For the next step, consider the functions
    \[\zeta_\epsilon(t)=\zeta(t)+\epsilon\tau(0,\xi_b(t)),\qquad \epsilon\in(0,1],\;t\in I.\]
    Observe that $\zeta_0=\zeta$ and $\zeta_1(t_{j,k})=\mathring\gamma(t_{j,k})$. Also, since $I\setminus\bigcup_jV_j$ is a set of measure zero,
    \[\int_{\bigcup_jV_j}\langle\Theta\circ\zeta_\epsilon,\zeta_\epsilon'\rangle dt=\int_{I}\langle\Theta\circ\zeta_\epsilon,\zeta_\epsilon'\rangle dt,\]
    independently of $\epsilon\in(0,1]$, and
    \[
        \lim_{\epsilon\searrow0}\int_{\bigcup_jV_j}
        \langle\Theta\circ\zeta_\epsilon,\zeta_\epsilon'\rangle dt
        =\int_{\bigcup_jV_j}
        \langle\Theta\circ\zeta,\zeta'\rangle dt.
    \]
    This gives
    \begin{equation*}
        \int_{I}\langle\Theta\circ\zeta_1,\zeta_1'\rangle dt %
        \leq \sum_{j=1}^\infty 2j
        \sum_{k=0}^{N_j}\sqrt{(\tilde\gamma_{d-1}(t_{j,k}+\ell_j)-\tilde\gamma_{d-1}(t_{j,k}))^2+(\tilde\gamma_d(t_{j,k}+\ell_j)-\tilde\gamma_d(t_{j,k}))^2}.
    \end{equation*}

    To finish the proof, write $I=[i_0,i_1]$ and join the points $\mathring\gamma(i_0)=\zeta_1(i_0)$ and $\mathring\gamma(i_1)=\zeta_1(i_1)$ with a curve $\alpha\colon[0,1]\to B^2(0,3/b)\setminus\{0\}$. 
     From the fact that $ \gamma$ is Lipschitz (with constant $\leq \sup_{x\in\domain}\|f_1(x)\|+\|f_2(x)\|\leq 3$), together with the estimate \eqref{eq:estimateballboxnew}, it follows that $\|\mathring\gamma(t)-\zeta_1(t)\|\leq \|\mathring\gamma(t)\|$.
    Thus we may apply the Dog-on-a-Leash Theorem \cite[Th.~3.11]{fulton2013algebraic} to the two curves $\nu_1$ and $\nu_2$ obtained by concatenating $\alpha$ with, respectively, $\mathring\gamma$ and $\zeta_1$.
   Since the part of the integral corresponding to $\alpha$ is the same for $\nu_1$ as for $\nu_2$,
    this implies that, in \eqref{eq:stepone} we have
    \[
        \int_{I}\langle\Theta\circ\mathring\gamma,\mathring\gamma'\rangle dt
        =\int_{I}\langle\Theta\circ\zeta_1,\zeta_1'\rangle dt,
    \]
    which shows item \ref{it:angleboundnew}. 
\end{proof}

\subsubsection{A priori bounds}
\label{sec:aprioribounds}
We will denote by $|X|$ the Lebesgue measure of a measurable set $X\subset\R^d$.

\begin{lemma}
    \label{lem:aprioribounds}
    Let $0<\epsilon<1$.
    Let $(\gamma,u)$ be a contender in \eqref{min:5}, which is to say that $\gamma\colon[-\ad,\ad]\to\overline\domain$ is a $\U$-horizontal curve with control $u\colon[-\ad,\ad]\to\U$, joining $\mathbf x_0=\gamma(-\ad)$ and $\mathbf x_1=\gamma(\ad)$. Assume that 
    \[\int_{-\ad}^{\ad}\convexL(\gamma(t),u(t))dt\leq 2a+\epsilon.\]
    Let $I=[i_0,i_1]$ be an interval with $\gamma_1(i_0)=-a\leq \gamma_1(t)\leq a=\gamma_1(i_1)$ for all $t\in I$.
    Then we have
    \[\sup_{t\in I}r(\gamma(t))\leq 5\,\epsilon^{1/4}\]
\end{lemma}
\begin{proof}
  We have
    \begin{equation*}
        \int_{I} \|u(t)\|_\infty dt
        \geq \int_{I} u_1 dt
        =\int_{I}\gamma'_1(t)dt
        =\gamma_1(i_1)-\gamma_1(i_0)
        =2a.
    \end{equation*}
    because $-1\leq u_1\leq 1$ for $u=(u_1,u_2)\in\U$.

    Let
    \[E=\{t\in I:r(\gamma(t))^2\geq \sqrt \epsilon\}.\]
    Then the Markov inequality \eqref{eq:markov} gives
    \[
        |E|
        \leq \frac1{\sqrt \epsilon} \int_{R_\gamma}r(\gamma(t))^2dt
        =\frac{1}{\sqrt\epsilon} \int_{R_\gamma}  \convexL(\gamma(t),u(t))-\|u(t)\|_\infty dt
        \leq\frac{1}{\sqrt\epsilon} (2a+\epsilon-2a)=\sqrt\epsilon.
    \]
    Going back to the definition \eqref{eq:defr} of $r$, we see that, since $\varphi^{-1}\circ \gamma$ is Lipschitz continuous with Lipschitz constant $\leq \sup_{x\in\domain}\|D\varphi^{-1}f_1(x)\|+\|D\varphi^{-1}f_2(x)\|\leq 4$, $r(\gamma(t))$ is also Lipschitz continuous with Lipschitz constant 4. %
    The bound we found for $|E|$ then implies
    \[\sup_{t\in R_\gamma}r(\gamma(t))\leq \sup_{t\in R_\gamma\setminus E}r(\gamma(t))+4|E|\leq \epsilon^{1/4}+4\sqrt\epsilon\leq 5\,\epsilon^{1/4}.\qedhere \]
\end{proof}

\subsubsection{Proof of the gap \texorpdfstring{\eqref{eq:hardone}}{}}
\label{sec:mininterior}

We will use the same notations as in the previous section. 

Let $d\geq 4$, $(\gamma,u)$ be a contender in \eqref{min:6}, and let $I=[i_0,i_1]$ be an interval with $\gamma(i_0)=-a\leq \gamma(t)\leq a\leq \gamma(i_1)$ for $t\in I$. 

By Lemma \ref{lem:aprioribounds}, we know that, if
\[\int_{-\ad}^{\ad}\convexL(\gamma,u)dt\leq 2a+\epsilon,\]
then 
\begin{equation}\label{eq:emptysets}
    V_j=\emptyset\qquad\text{for all}\qquad j<\frac{1}{5\epsilon^{1/4}}.
\end{equation}
Assume that $\epsilon>0$ is such that $5^{5/2}\epsilon^{5/8}<s_0$, %
with $s_0$ as in Lemma \ref{lem:ballboxapplied},
so  that $V_j$ may be nonempty only if 
\begin{equation}\label{eq:jconstraint}
    j^{-5/2}\leq (5\epsilon^{1/4})^{5/2}<s_0.
\end{equation}
This will ensure that we can apply Lemma \ref{lem:partitionestimatesnew}\ref{it:arclengthnew}.

By Lemmas \ref{lem:mustturn} and \ref{lem:partitionestimatesnew} we know that
\begin{multline*}
    ab-2\pi
    \leq\left| \int_{I}\langle \Theta\circ\mathring\gamma,\mathring\gamma'\rangle dt\right|%
    \leq \sum_{j=1}^{\infty}2j\sum_{k=0}^{N_j}\|\tau\circ\tilde\gamma(t_{j,k}+\ell_j)-\tau\circ\tilde\gamma(t_{j,k})\|%
    \\
    \leq \sum_{j=1}^\infty C%
    j\frac{|V_{j-1}|+|V_j|+|V_{j+1}|}{j^{\frac52(d-3)}}%
    \leq \sum_{j=1}^\infty \frac{3(2a)C%
    }{j^{\frac52(d-3)-1}}%
    \leq \sum_{j=1}^\infty \frac{6aC%
    }{j^{3/2}}.%
\end{multline*}
Using \eqref{eq:emptysets} and \eqref{eq:jconstraint} again, we see that these estimates can be rewritten as
\[ab-2\pi\leq %
\sum_{j\geq 1/5\epsilon^{1/4}} \frac{6aC%
}{j^{3/2}},\]
which gives a constraint on how small $\epsilon$ can be. This proves \eqref{eq:hardone}.
\qed
 
\section{Regularization of the Lagrangian and proof of Theorem \ref{thm:example}}
    \label{sec:regularityL}

     The Lagrangian density $\nonconvexL$ defined in \eqref{eq:deflagrangians} is not very regular; not only are there numerous points where it is not differentiable, but it is also discontinuous at $x_1=\pm a$.
     However this lack of regularity is not crucial to the obtainment of results analogous to those of Theorem \ref{thm:gaps}. Here we show how to obtain Theorem \ref{thm:example} from Theorem \ref{thm:gaps} using a smoothing technique.

    \begin{theorem}\label{thm:smoothlagrangian}
        In the situation described in Section \ref{sec:setting}, and under Assumptions \ref{asm:deltaGamma}, \ref{ass:nonemptycontenders}, and  \ref{asm:ballboxonOmega}, there is a $C^\infty$ Lagrangian density $\normalL\colon\domain\times U\to\R$ such that
        \[\mincurves (\normalL,U)>\minarcs(\normalL,U),\]
        where these are defined as in \eqref{min:1} and \eqref{min:2} with $\nonconvexL$ replaced by $\normalL$.
    \end{theorem}

    Theorem \ref{thm:example} is Theorem \ref{thm:smoothlagrangian} taking %
    the corresponding objects in Theorem \ref{thm:smoothlagrangian} 
    appropriately transformed by the diffeomorphism defined in Proposition \ref{prop:gammadelta}.
    
    \begin{proof}
    Take a $C^\infty$ mollifier function $\psi\colon \R^d\times\R^2\to[0,+\infty)$ with $\int_{\R^2}\int_{\R^{d}}\psi(x,u)\,dx\,du=1$ and supported in the unit ball $B^{d+2}(0,1)$ for $0<\epsilon<(1-\lambda)a$. Let also $\psi_\epsilon(x)=\epsilon^{-d-2}\psi(x/\epsilon)$, so that $\int_{\R^{d}}\int_{\R^2}\psi_\epsilon(x,u)\,dx\,du=1$ and $\psi_\epsilon$ is supported in the ball $B^{d+2}(0,\epsilon)$. The convolution
    \[\bar L_\epsilon(x,u)=(\psi_\epsilon*\nonconvexL)(x,u)=\int_{\R^2}\int_{\R^d}\psi_\epsilon(y,v)\nonconvexL(x-y,u-v)\,dy\,dv,\qquad x\in\R^d,\;u\in\R^2,\]
    is a $C^\infty$ function $\bar L_\epsilon\colon \R^d\times\R^2\to\R$ satisfying, for some $C(K)>0$ depending on the compact set $K\subseteq \overline{\domain}\times \R^2$, 
    \begin{align*}
    &|\bar L_\epsilon(x,u)-\nonconvexL(x,u)|\leq C(K)\epsilon&& \text{for $(x,u)\in K\subseteq \overline\domain\times \R^2$ with $|x_1\pm a|>\epsilon$},\\
    &0\leq \bar L_\epsilon(x,u)\leq \sup_{y\in\overline\domain}\nonconvexL(y,u) && \text{for $(x_1,u)\in \overline\domain\times \R^2$ with either $|x_1- a|\leq \epsilon$ or $|x_1+a|\leq \epsilon$}.
    \end{align*}
    The first and second estimates are true in general for any approximation of a continuous function, and the second one follows from the fact that $\psi(x,u)\,dx\,du$ is a probability measure.
    Thus, for $\eta$ as in \eqref{eq:def-eta} and $\nu_t=\frac12(\delta_{(1,1)}+\delta_{(1,-1)})$, we have (cf. Lemma \ref{lem:2})
    \[\minarcs(\bar L_\epsilon,U)\leq\int_{-\ad}^{\ad}\int_{U}\bar L_\epsilon(\eta(t),u)\,d\nu_t(u)\,dt\leq 2a+2\ad C(K)\epsilon =\minarcs(\nonconvexL,U)+2\ad C(K)\epsilon.\]
    On the other hand, for a horizontal curve $\gamma$ with control $u(t)\in U$, we have %
    \begin{equation*}
    \lim_{\epsilon\searrow0}\int_{-\ad}^{\ad} \bar L_\epsilon(\gamma(t),u(t))dt= \int_{-\ad}^{\ad}\nonconvexL(\gamma(t),u(t))\,dt.
    \end{equation*}
    In particular, this means that $\lim_{\epsilon\searrow 0} \mincurves(\bar L_\epsilon,U)=\mincurves(\nonconvexL,U)$. This, in turn, means that given $z>0$ there is $\epsilon_0>0$ such that, for all $0<\epsilon<\epsilon_0$, $\mincurves(\nonconvexL,U)\leq \mincurves(\bar L_\epsilon,U)+z$.

    By taking $\epsilon$ and $z$ small, and applying \eqref{eq:minmain}, we can ensure that
    \begin{equation}\label{esti_reg}
    \mincurves(\bar L_\epsilon,U)\geq \mincurves(\nonconvexL,U)-z>\minarcs(\nonconvexL,U)+2\ad C(K)\epsilon\geq \minarcs(\bar L_\epsilon,U).
    \end{equation}
    This proves the theorem with $\normalL=\bar L_\epsilon$.
    \end{proof}

\section{Gap for a Mayer problem and proof of Theorem \ref{thm:example2}}\label{sec:vinterlike}

By the extension procedure described in Appendix \ref{sec:nullLagrangian}, we obtain Theorem \ref{thm:secondexample} as a natural consequence of Theorem \ref{thm:smoothlagrangian}.

 \begin{theorem}[] \label{thm:secondexample}
    Let $U\subset \R^2$ be a given convex set that contains $\U$.    
    For $d\geq 5$, there exist an open set $\domain_2\subset\R^d$, a terminal cost function $\tilde\g$, a controlled vector field $\tilde \f$, a point $\tilde {\mathbf x}_0\in\domain_2$, and a target set $\destination_2$ %
    with the following properties (fixing the null running cost
    $\tilde\normalL=0$ and the terminal cost $\tilde\g$ for all the problems, omitted for readability):
    \begin{itemize}
        \item We have
        \begin{equation}\label{eq:secondexamplegap}
            \mincurves^{U} \!\left(\overline\Omega_2\right)>%
            \minarcs^{U}\!\left(\overline\Omega_2\right),%
        \end{equation}
        \item 
           However, if we do not restrict to the set $\overline\domain_2$, we have \begin{equation}\label{eq:secondexamplenogap}\mincurves^{U}(\R^{d})%
           =\minarcs^{U}(\R^{d}),%
        \end{equation}
        \item $\tilde\g$ is $C^\infty$,
        \item $\domain_2$ is a domain with $C^\infty$ boundary and diffeomorphic to the open ball of $\R^{d}$,
        \item $\destination_2\subset\domain_2$ is a target set satisfying \ref{it:targetconstraint},
        \item $\tilde f$ is $C^\infty$ and, for $t\in [0,T]$ and $x\in \overline\Omega_2$, the sets $\tilde f(t,x,U)$ are not convex. %
    \end{itemize}
\end{theorem}
To obtain Theorem \ref{thm:example2} from Theorem \ref{thm:secondexample}, apply the transformation from Proposition \ref{prop:gammadelta}, and do the obvious notational substitutions: $\domain,g,\initial,X$ in Theorem \ref{thm:example2} are $\domain_2,\tilde g,\tilde{\mathbf x}_0,X_2$ in Theorem \ref{thm:secondexample}, respectively.

\begin{proof}[Proof of Theorem \ref{thm:secondexample}] 
\label{proof:secondexample}
    Take all the definitions from Sections \ref{sec:setting}--\ref{sec:regularityL} in dimension $d-1$. In particular, the domain $\domain$ is diffeomorphic to the ball of $\R^{d-1}$. In particular $\normalL=\bar L_\epsilon$ is the smooth running cost from Theorem \ref{thm:smoothlagrangian}.

     Apply Lemma \ref{lem:prelimequivalence} in the situation $P=(\domain,\initial,\destination,\T,U,\normalL,0,\f)$ to obtain a corresponding situation $\tilde P=(\tilde\domain,\tilde{\mathbf x}_0,\tilde\destination,\T,U,0,\tilde g,\tilde f)$ such that $\tilde g$ is $C^\infty$, $\tilde \domain=\domain\times \R$ is of dimension $d$,
    \begin{equation}\label{eq:previous}\mincurves^{\normalL,0,U}(\domain,\destination)=\mincurves^{0,\tilde\g,U}(\tilde\domain,\tilde\destination),\quad\text{and}\quad 
    \minarcs^{\normalL,0,U}(\domain,\destination)=\minarcs^{0,\tilde\g,U}(\tilde\domain,\tilde\destination).
    \end{equation}
    Observe that, since $\normalL$ is not convex and since $\tilde f=(f,\normalL)$ (cf.~proof of Lemma \ref{lem:prelimequivalence}), the sets $\tilde f(t,x,U)$ are not convex. 

    Let $I$ be a bounded, open interval containing $[0,\mincurves^{\normalL,0,U}(\domain,\destination)]$.
    Then $\domain\times I\subset\tilde\domain$ contains $\tilde{\mathbf x}_0$ and an $\tilde f$-admissible curve joining $\tilde{\mathbf x}_0$ and $\tilde\destination$. 
    Take $\domain_2$ to be any subset of $\tilde\domain=\domain\times\R$ containing $\domain\times I$ and diffeomorphic to a ball. Take also $X_2$ to be any subset of $\tilde \destination$ containing $\destination\times I$. %
    From our choice of $I$, we see that, as any minimizing sequence of $\mincurves^{0,\tilde\g,U}(\tilde\domain,\tilde\destination)$ and $\minarcs^{0,\tilde\g,U}(\tilde\domain,\tilde\destination)$ will eventually involve contenders whose integral curves are completely contained in $\Omega\times I$, we have:
    \[\mincurves^{0,\tilde\g,U}(\tilde\domain,\tilde\destination)=
    \mincurves^{0,\tilde\g,U}(\domain_2,\destination_2)\qquad\text{and}\qquad
    \minarcs^{0,\tilde\g,U}(\tilde\domain,\tilde\destination)=\minarcs^{0,\tilde\g,U}(\domain_2,\destination_2).\]
    Together with \eqref{eq:previous}, this gives
    \[\mincurves^{\normalL,0,U}(\domain,\destination)=\mincurves^{0,\tilde\g,U}(\domain_2,\destination_2)\qquad\text{and}\qquad
      \minarcs^{\normalL,0,U}(\domain,\destination) =\minarcs^{0,\tilde\g,U}(\domain_2,\destination_2).\]
    From \eqref{eq:minmain} in Theorem \ref{thm:gaps} (translating time by $\ad$, so that $[-\ad,\ad]$ becomes $[0,\T]$), we have
    \[\mincurves^{\normalL,0,U}(\domain,\destination) =\mincurves(\normalL,U)
    >\minarcs(\normalL,U) =\minarcs^{\normalL,0,U}(\domain,\destination).\]
    We conclude that
    \[
    \mincurves^{0,\tilde\g,U}(\domain_2,\destination_2)
    >\minarcs^{0,\tilde\g,U}(\domain_2,\destination_2).
    \]
    This, %
    together with Lemma \ref{lem:convexequiv} implies \eqref{eq:secondexamplegap}. %

    To show \eqref{eq:secondexamplenogap}, observe that from the H\"older regularity of the sub-Riemannian distance \cite[Ch.~10]{ABB20} and the constructions in Section \ref{sec:setting} in dimension $d-1$, it follows that
    \begin{equation}\label{eq:someequality}\mincurves^{\normalL,0,U}(\R^{d-1},X)=\minarcs^{\normalL,0,U}(\R^{d-1},X)
    \end{equation}
    because ---absent the obstacle caused by the state constraint $\domain$--- the minimizer $\eta$ can be approximated using curves $\gamma_1,\gamma_2,\dots$ verifying $\gamma'_i(t)=f_1(\gamma_i(t))=(1,0,\dots,0)$ for $t$ in a subset of $[0,T]$ of measure  $\geq T-1/i$, which also join $\gamma_i(0)=\mathbf x_0$ to the target set $\gamma_i(T)\in\destination $; these curves will have nearly-minimal cost. Reasoning as above, but this time using \eqref{eq:someequality} instead of \eqref{eq:minU} and $\R^{d-1}$ instead of $\domain$, we conclude that \eqref{eq:secondexamplenogap} holds. 
 \end{proof}

\section{Gap for problems with no terminal constraints}
\label{sec:targetset}

We now take up the discussion of the role of the target set $X$ in the relaxation gap properties.
In order to illustrate this fact, we prove it is possible in our case to remove the endpoint constraint $X$ and replace it by a penalization $g$ in a terminal cost presenting the same gap property, allowing to obtain the following extensions of Theorems \ref{thm:example} and \ref{thm:example2}, corresponding to Theorems \ref{thm:example4} and \ref{thm:example3}.

\begin{theorem}\label{gap_withoutX}
\begin{enumerate}[label=\roman*.,ref=(\roman*)]
    \item\label{it:part1} Let $\Omega$ be the ball in $\R^d$, $d\ge 4$. There exist a $C^\infty$ Lagrangian $\normalL$, a $C^\infty$ controlled vector field $\f$ that is convex in the controls $u\in U$, a $C^\infty$ terminal cost $g$ and a convex control set $U\subseteq \R^2$ such that, for $X=\overline\Omega$ (i.e. there is no terminal constraint),
   the associated costs %
    satisfy 
    \begin{gather*}
        \mincurves^{\normalL,g,U}(\overline{\domain},\overline\domain) >\minarcs^{\normalL,g,U}(\overline{\domain},\overline\domain),\\
        \mincurves^{\normalL,g,U}(\R^d,\R^d)%
        =\minarcs^{\normalL,g,U}(\R^d,\R^d).
    \end{gather*}
\item\label{it:part2}  Let $\Omega$ be the ball in $\R^d$, $d\ge 5$, and $U$ be any subset of $\R^2$ containing $[-1,1]\times[-1,1]\subseteq U$. There is a smooth  terminal cost $g$, a $C^\infty$ controlled vector field $\f$, a point $\initial\in \domain$, and (taking also $L=0$ and $X=\overline \domain$),
\begin{gather*}
    \mincurves^{0,g,U}(\overline{\domain},\overline\domain) >\minarcs^{0,g,U}(\overline{\domain},\overline\domain),\\
    \mincurves^{0,g,U}(\R^d,\R^d)%
    =\minarcs^{0,g,U}(\R^d,\R^d).
\end{gather*}
\end{enumerate}
\end{theorem}
With slightly different notations, Theorem \ref{thm:example4} is part \ref{it:part1} and
Theorem \ref{thm:example3} is part \ref{it:part2}.

\begin{remark}\label{rk:remove_initial}
Note that in Theorem \ref{gap_withoutX}, the initial point is fixed, i.e. $\gamma(0)=\initial\in \R^d$. In accordance with Section~\ref{sec:role_in_tar}, the same gap holds when the latter condition is replaced by the constraint $\gamma(0)\in X_0$ where $X_0\subset \R^d$ is an open subset of $\Gamma\cap\{x_1<-a-\epsilon\}$.
    In this case, a similar penalization procedure as the one used in the proof of Theorem \ref{gap_withoutX} provides the same gap result when removing the initial constraint $X_0$ as well, via its incorporation in a supplementary initial time cost function involving $\gamma(0)$.
\end{remark}

In order to prove Theorem~\ref{gap_withoutX}, we keep the setting of Sections~\ref{sec:setting}--\ref{sec:regularityL} with the same notations, especially the same constraint set $\Omega$ and $C^\infty$ Lagrangian $\normalL=\bar L_\epsilon$, and convex control set $U$ containing $[-1,1] \times[-1,1]$.
Fix an open set $X$ such that $X\subseteq \Gamma\cap\{x\in\R^d:x_1>a+\epsilon\}$, and denote its complement in $\overline \domain$ by $X^c=\overline\domain\setminus X$, which is closed. For $\alpha>M_{\mathrm{c}}(\normalL,U)$, where $M_{\mathrm{c}}(\normalL,U)>0$ is as in \eqref{min:1}.
 Take a $C^\infty$ mollifier function $\varphi\colon \R^d\to[0,+\infty)$ with $\int_{\R^{d}}\varphi(x)\,dx=1$ and supported in the unit ball $B^{d}(0,1)$ for $0<\epsilon<(1-\lambda)a$. Let also $\varphi_\epsilon(x)=\epsilon^{-d}\psi(x/\epsilon)$, so that $\int_{\R^{d}}\varphi_\epsilon(x)\,dx=1$ and $\varphi_\epsilon$ is supported in the ball $B^{d}(0,\epsilon)$. 
 Let $\chi_{X^c}$ be the indicator function of $X^c$, and
define the $C^\infty$ convolution $g_\epsilon: \R^d\to \R$ as
   \[ g_\epsilon(x)=(\varphi_\epsilon*\chi_{X^c})(x)=\int_{\R^d}\varphi_\epsilon(y) \chi_{X^c}(x-y)\,dy\,\qquad x\in\R^d,\]
satisfying $0\leq g_\epsilon(x)\leq 1$ for every $x\in \R^d$ and $ g_\epsilon(x)=1$  for every $x\in X^c$, and $g_\epsilon(x)=0$ for $x\in X\setminus E_{\epsilon}$, with $E_\epsilon=\{x\in X \mid d(x,\partial X)< c\epsilon\}$, where $c>0$ is independent of $\epsilon>0$.

Define the following penalized optimization problems:
    \begin{itemize}

                    \item Non-convex Lagrangian and convex set of controls, optimization over measurable controls inducing horizontal curves: 
            \begin{customopti}%
                {inf}{u}
                {\int_{-\ad}^{\ad} \normalL(\gamma(t),u(t))\,dt+\alpha g_\epsilon(\gamma(\ad))}{\label{min:1p}}{%
                \mincurves^p(\normalL,U)=}
                \addConstraint{u\colon[-\ad,\ad]\to U\;\text{measurable}}
                \addConstraint{\gamma'(t)=f(t,\gamma(t),u(t))\text{ a.e.~$t$}}
                \addConstraint{\gamma(t)\in\overline\domain\;\text{for all $t$}}
                \addConstraint{\gamma(-\ad)=\initial.}
            \end{customopti}
        \item Non-convex Lagrangian and convex set of controls, optimization over Young measures:
            \begin{customopti}%
                {inf}{(\nu_t)_{t\in[-\ad,\ad]}}
                {\int_{-\ad}^{\ad}\int_{U} \normalL(\gamma(t),u)\,d\nu_t(u)\,dt+\alpha g_\epsilon(\gamma(\ad))}{\label{min:2p}}{%
                \minarcs^p(\normalL,U)=}
                \addConstraint{\text{$(\nu_t)_t$ is a Young measure on $U$}}
                \addConstraint{\text{$\gamma'(t)=\int_Uf(t,\gamma(t),u)\,d\nu_t(u)$ a.e.~$t$}}
                \addConstraint{\gamma(t)\in\overline\domain\;\text{for all $t$}}
                \addConstraint{\gamma(-\ad)=\initial.}
            \end{customopti}
            \end{itemize}

\begin{proof}[Proof of Theorem \ref{gap_withoutX}]
We prove the first statement \ref{it:part1} of the theorem. The proof of the second one \ref{it:part2} is omitted as it is obtained by the same penalization argument, but this time in the framework of a Mayer problem, starting from the gap result stated in Theorem \ref{thm:secondexample}.

We start by showing that $\minarcs^p(\normalL,U)\leq 2a<M_{\mathrm{c}}^p(\normalL,U)$.
Since $\bar g_\epsilon(x)=1$  for every $x\in X^c$ and $\alpha>M_{\mathrm{c}}(\normalL,U)>2a$, valid contenders $u$ (respectively, $(\nu_t)_t$) for the latter problems \eqref{min:1p} and \eqref{min:2p} satisfy necessarily $\gamma(T)\in X$.
Indeed, assuming that $(\nu_t)_{t\in[-\ad,\ad]}$ is a contender for Problem \eqref{min:2p} such that $\gamma(T)\notin X$, it is easy to see that by non-negativity of $\normalL$ we have \[\int_{-\ad}^{\ad}\int_{U} \normalL(\gamma(t),u)\,d\nu_t(u)\,dt+\alpha\chi_{X^c}(\gamma(T))=\int_{-\ad}^{\ad}\int_{U} \normalL(\gamma(t),u)\,d\nu_t(u)\,dt+\alpha>2a.\] The same argument holds for Problem~\eqref{min:1p}.
Now consider $(\nu_t)_{t\in[-\ad,\ad]}$ be a $U$-admissible Young measure such that \[\int_{-\ad}^{\ad}\int_{U} \normalL(\gamma(t),u)\,d\nu_t(u)\,dt\leq 2a\] and $\gamma(T)\in X$. Note that its existence is guaranteed by considering $\eta$ as in \eqref{eq:def-eta} and $\nu_t=\frac12(\delta_{(1,1)}+\delta_{(1,-1)})$. Then there exists $\epsilon_0>0$ such that we have $g_\epsilon(\gamma(\ad))=0$ for every $\epsilon \in (0,\epsilon_0)$, and it follows that \[\int_{-\ad}^{\ad}\int_{U} \normalL(\gamma(t),u)\,d\nu_t(u)\,dt+\alpha g_\epsilon(\gamma(\ad))\leq 2a.\]
           Noticing that for a contender $u$ for~\eqref{min:1p} (satisfying necessarily $\gamma(T)\in X$), we have \[ \int_{-\ad}^{\ad}\int_{U} \normalL(\gamma(t),u(t))\,dt+\alpha g_\epsilon(\gamma(\ad))\geq M_{\mathrm{c}}(\normalL,U)>2a,\] and we obtain that $M_{\mathrm{c}}^p(\normalL,U)>2a$.

Now we apply the regularization procedure of the Lagrangian made in Section~\ref{sec:regularityL}, which provides a regularized Lagrangian $\bar L_\epsilon$ of $\normalL$ for every $\epsilon>0$. Considering the associated penalized costs $\minarcs^p(\bar L_\epsilon,U)$ and $\mincurves^p(\bar L_\epsilon,U)$, we obtain, as in~\eqref{esti_reg},
the existence of $\epsilon_1\in (0,\epsilon_0)$ such that for every $\epsilon \in (0,\epsilon_1)$,
 \[\mincurves^p(\bar L_\epsilon,U)>\minarcs^p(\normalL,U)+2\ad C(K)\epsilon\geq \minarcs^p(\bar L_\epsilon,U),\] where $C(K)>0$ is defined as in Section~\ref{sec:regularityL}.
 The proof of the theorem can be concluded by considering $\normalL=\bar L_\epsilon$ and $g=g_\epsilon$, for a given $\epsilon\in (0,\epsilon_1)$.
\end{proof}
 
\section{Failure of the Filippov-Wa\v zewski approximation in the presence of state constraints}
\label{sec:FW}
Let $\Omega\subset\R^d$ be a bounded open set  with smooth boundary, and define for $x\in \partial \Omega$, the vector $n(x)$ of $\R^d$ as the internal unit normal vector to $\partial \Omega$ at the point $x$. 
For a controlled vector field $f$ and a control set $\controls$,
define the set-valued map $\F(t,x)$ given by
\begin{equation}
    \label{eq:defF}\F(t,x)=\{f(t,x,u):u\in \controls\},\qquad t\in [0,\T], \;x\in \overline \domain,
\end{equation}
 for $x \in \overline \Omega$,
and
\begin{align*}
    \mathcal{S}_{\overline{\domain}}(x)=&\{\gamma\colon[0,T]\to\overline{\domain} \, \ \text{absolutely continuous} \mid \ \gamma(0)=x, \ \gamma'(t)\in \F(t,\gamma(t))\},\\   
    \mathcal{S}_{\overline{\domain}}^r(x)=&\{\gamma\colon[0,T]\to\overline{\domain} \, \ \text{absolutely continuous} \mid \ \gamma(0)=x, \ \gamma'(t)\in \operatorname{conv}\F(t,\gamma(t))\}.
\end{align*}

\begin{theorem}{\cite[Corollary 3.2]{frankowska2000filippov}}
For $d\geq 1$, let $U\subset \R^2$ be compact, and $\f$ be a smooth controlled vector field.
Assume that there exists $\delta>0$ such that \begin{equation}\label{eq:IPC}\max_{v\in F(t,x)}\langle v , n(x)\rangle > \delta,\end{equation} for every $x\in \partial \Omega$.
 Then the set $\mathcal{S}_{\overline{\domain}}(x)$ is dense in $\mathcal{S}_{\overline{\domain}}^r(x)$ for the metric of uniform convergence, for any $x \in \overline \domain$.

\end{theorem}

This contrasts with the following observation.
Indeed, using a contrapositive argument, we obtain the following result as a corollary of Theorem \ref{gap_withoutX}\ref{it:part2}.%
This suggests that some supplementary higher order conditions have to be imposed in order to guarantee the same density result when the inward-pointing condition \eqref{eq:IPC} is violated.

\begin{corollary}[of Theorem \ref{gap_withoutX}\ref{it:part2}]
    \label{coro:noFW}
For $d\geq 5$, there exist a bounded open set $\domain \subset\R^d$ with smooth boundary, a non-convex compact set of controls $U\subset\R^2$, a smooth controlled vector field $\f(x,u)$, a point $\initial \in\domain$, such that $\mathcal{S}_{\overline{\domain}}(\initial)$ is not dense in $\mathcal{S}_{\overline{\domain}}^r(\initial)$ for the metric of uniform convergence.

In fact, more is true: there exist $\delta,\varepsilon>0$ and a curve $\gamma$ in $\mathcal S^{\mathrm r}_{\overline\Omega}(\initial)\setminus \mathcal S_{\overline\Omega}(\initial)$ such that, for all $x\in\R^d$ with $\|x-\initial\|\leq \delta$ and all $\theta\in  \mathcal S_{\overline\Omega}(x)$
\[\|\gamma-\theta\|_\infty\geq\varepsilon.\]
\end{corollary}
This result contains Theorem \ref{thm:FW-intro}.
\begin{proof}
    First note that, by the discussion of Section \ref{sec:role_in_tar} (cf.~Remark \ref{rk:remove_initial}), the statement of Theorem \ref{gap_withoutX}\ref{it:part2} remains valid when the initial set is replaced by a small ball $B(\initial,\delta)$ around $\initial$. Then, since $g$ is a continuous function, we see that, if it were possible to approximate the integral curve $\tilde\eta$ corresponding to the Young measure that minimizes the relaxed problem, in the supremum norm, with a sequence of admissible curves $\theta_i$, we would be able to get their endpoints arbitrarily close, $\|\theta_i(T)-\tilde\eta(T)\|\to 0$, and hence we would have $g(\theta_i(T))-g(\tilde\eta(T))\to 0$, contradicting the existence of the gap found in the statement of the theorem. The gap instead means that any curves $\theta$ starting in $B(\initial,\delta)$  must end at a distance $\|\theta(T)-\tilde\eta(T)\|\geq\varepsilon>0$. Taking $\gamma=\tilde\eta$, we obtain the corollary.
\end{proof} 

\appendix
\section{On the equivalence of certain problems}
\label{sec:prelims}
\subsection{Transforming between Bolza and Lagrange problems into Mayer problems}
\label{sec:nullLagrangian}

In this section we include two well-known lemmas that allow for useful reformulations of problems \eqref{min:orig}--\eqref{min:relaxed}. %

The following lemma, which we will prove for completeness, is well known in the community. It states that every optimization problem $P$ of type \eqref{min:orig}--\eqref{min:relaxed} %
has an equivalent counterpart $\tilde P$ with null Lagrangian density $\tilde\normalL=0$.
\begin{lemma}
    \label{lem:prelimequivalence}
    Denote by $P=(\domain,\initial,\destination,\T,\controls,\normalL,\g,\f)$ the situation determined by a domain $\domain\subseteq\R^d$, an initial point $\initial\in\domain$, a target set $\destination\subset\domain$, an amount of time $\T>0$, a set of controls $\controls\subset\R^m$, a Lagrangian density $\normalL\colon\R\times\domain\times\controls\to\R$, an initial-final cost  function $\g\colon\R^d\times\R^d\to\R$, and a controlled vector field $\f\colon\R\times\domain\times\controls\to\R^d$. There exists a corresponding situation $\tilde P=(\tilde \domain,\tilde {\mathbf x}_0,\tilde\destination,\T,\controls,0,\tilde\g,\tilde\f)$ determined by corresponding objects $\tilde\domain,\tilde {\mathbf x}_0,\tilde\destination,\tilde\g,\tilde\f$ (and $\tilde \T=\T$, $\tilde\controls=\controls$,  $\tilde\normalL=0$) such that
    \begin{gather}
        \label{eq:equivcurves}\mincurves^{\normalL,\g,\controls}(\domain,\destination)=\mincurves^{0,\tilde\g,\tilde\controls}(\tilde\domain,\tilde\destination),\\
        \label{eq:equivarcs}\minarcs^{\normalL,\g,\controls}(\domain,\destination)=\minarcs^{0,\tilde\g,\tilde\controls}(\tilde\domain,\tilde\destination).%
    \end{gather}
    There is $A_0$ such that there is a one-to-one correspondence between the sets of contenders $x$ admissible in the situation $P$ and the contenders $\tilde x$ admissible in $\tilde P$ (which can be curves, Young measures, or occupation measures) with $A_\normalL(\tilde x)\leq A_0$, where $A_\normalL$ indicates the corresponding integral of $\normalL$, as indicated in the corresponding problems \eqref{min:orig} and \eqref{min:relaxed}. 
    Moreover, given two controls $u_1$ and $u_2$ contending in \eqref{min:orig}, with integral curves $\gamma_1$ and $\gamma_2$, satisfying  $\gamma_i(t)=\initial+\int_0^tf(t,\gamma_i(t),u_i(t))dt$, the corresponding curves $\tilde\gamma_1$ and $\tilde\gamma_2$, satisfying $\tilde\gamma_i(t)=\tilde{\mathrm x}_0+\int_0^t\tilde f(t,\tilde\gamma_i(t),u_i(t))dt$, verify
    \begin{gather*}
        \|\gamma_1(t)-\gamma_2(t)\|\leq \|\tilde\gamma_1(t)-\tilde\gamma_2(t)\|,\qquad t\in[0,\T],\\
        \left|\int_0^TL(t,\gamma_1(t),u_1(t))dt-\int_0^TL(t,\gamma_2(t),u_2(t))dt\right|\leq \|\tilde\gamma_1(\T)-\tilde\gamma_2(\T)\|.
    \end{gather*}
    Analogous relations hold for contenders $(\nu^1_t)_{t\in[0,T]}$ and $(\nu^2_t)_{t\in[0,T]}$ of \eqref{min:relaxed} and their integral curves.
    Finally, if $\g$ is $C^\infty$, so is $\tilde\g$.
\end{lemma}
\begin{proof}
    Let $\pi\colon\R^{n+1}\to\R^d$ be the projection given by $\pi(x_1,\dots,x_{n+1})=(x_1,\dots,x_n)$ and
    \begin{align*}
        \tilde \domain&=\domain\times\R\subseteq\R^{n+1},%
        &\tilde{\mathbf x}_0&=(\initial,0)\in \R^{n+1},\\
        \tilde \destination&=\destination\times\R,%
        &\tilde \T&=\T,\\
        \tilde\controls&=\controls,%
        &\tilde \normalL(t,x,u)&=0,\\
        \tilde \g(x)&=x_{n+1}+\g(\pi(x)),%
        &\tilde \f(t,x,u)&=(\f(t,\pi(x),u),\normalL(t,\pi(x),u)).
    \end{align*}
    Observe that if a curve $\tilde\gamma$ satisfies $\tilde\gamma(0)=\tilde{\mathbf x}_0$ and $\tilde\gamma'(t)=\tilde\f(t,\tilde\gamma(t),u(t))$ for almost every $t$, then the last coordinate of $\tilde\gamma$,
    \[\tilde\gamma_{n+1}(t)=\int_0^t\normalL(s,\pi(\tilde\gamma(s)),u(s))\,ds,\]
    keeps track of the integral of $\normalL$. To a contender $u$ for \eqref{min:orig} in situation $P$ corresponds the same contender $u$ in situation $\tilde P$, and the curve $\tilde \gamma$ ends up being given by
    \begin{equation}\label{eq:liftgamma}
        \tilde\gamma(t)=(\gamma(t),\int_0^tL(s,\gamma(s),u(s))\,ds),\qquad t\in[0,\T],
    \end{equation}
    and also corresponds bijectively to the corresponding curve $\gamma$ as $\gamma=\pi\circ\tilde\gamma$. The last coordinate is completely determined by the definition of the controlled vector field $\tilde\f$. Moreover, the cost also coincides:
    \begin{multline*}
        \int_0^\T\tilde\normalL(s,\tilde\gamma(s),u(s))\,ds+\tilde\g(\tilde\gamma(\T))=0+\tilde\g(\tilde\gamma(\T))
        =\tilde\gamma_{n+1}(\T)+g(\pi(\tilde\gamma(\T))) \\
        =\int_0^\T \normalL(s,\gamma(s),u(s))\,ds+g(\gamma(\T)).
    \end{multline*}
    The estimates claimed at the end of the statement of the lemma follow immediately from  the definitions.
\end{proof}

\subsection{The convexification of the velocities is equivalent to the Young measure relaxation}
\label{sec:convexificaton}

A fundamental technique in optimal control is the translation of the problem into a variational one. This is achieved by exchanging the controlled vector field $f(t,x,u)$ by a differential inclusion constraint associated to the set-valued map $\bar \F(t,x)$ given by \eqref{eq:defF}.
The Lagrangian density $\normalL(t,x,u)$ is also replaced by a two functions on the tangent bundle: $\bar\normalL$ and  $\var\normalL(t,x,v)$, defined by
\begin{equation}
        \label{def:barL}
    \bar\normalL(t,x,v)=\inf\{ \normalL(t,x,u): f(t,x,u)=v,\;u\in \controls\},
\end{equation} 
and
\begin{equation}
   \label{def:varL}
    \var\normalL(t,x,v)=\sup\{\ell(v)\;|\;\ell\colon\R^d\to\R\;\text{linear},\;\ell(f(t,x,u))\leq \normalL(t,x,u)\;\forall u\in \controls\},
\end{equation}
for $t\in[0,\T]$, $x\in\domain$, $v\in\R^d$. Observe that $\var\normalL$ is the convexification of $\bar\normalL$.
With these definitions, the equivalent of problem \eqref{min:orig} is
\begin{customopti}%
                {inf}{\gamma}
                {\int_{0}^T \bar\normalL(t,\gamma(t),\gamma'(t))\,dt+\g(\gamma(T))}%
                {\label{min:variational}}%
                {%
                \variational^{\bar\normalL}=\variational^{\bar\normalL,\g}({\domain},\destination)=}
                \addConstraint{\gamma\colon[0,T]\to\domain\;\,\text{absolutely continuous}}
                \addConstraint{\gamma'(t)\in \bar\F(t,\gamma(t)),\;\text{a.e.}\;t\in[0,\T]}
                \addConstraint{\gamma(0)=\initial,\;
                \gamma(\T)\in\destination,}
\end{customopti}
and its convex relaxation is
\begin{customopti}%
                {inf}{\gamma}
                {\int_{0}^T \var\normalL(t,\gamma(t),\gamma'(t))\,dt+\g(\gamma(T))}%
                {\label{min:variationalrelaxed}}%
                {%
                \variationalrelaxed^{\var\normalL}=\variationalrelaxed^{\var\normalL,\g}({\domain},\destination)=}
                \addConstraint{\gamma\colon[0,T]\to\domain\;\,\text{absolutely continuous}}
                \addConstraint{\gamma'(t)\in \operatorname{conv}\bar\F(t,\gamma(t)),\;\text{a.e.}\;t\in[0,\T]}
                \addConstraint{\gamma(0)=\initial,\;
                \gamma(\T)\in\destination.}
\end{customopti}
Here, $\operatorname{conv}\bar\F(t,x)$ is the closed convex hull of $\bar\F(t,x)$. 

\begin{lemma}
    \label{lem:convexequiv}
    Fixing $\domain, \initial, \destination, \normalL, \g,\f$ as above, and letting $\bar\normalL$, $\var\normalL$, and $\bar\F$ be as in \eqref{def:barL}, \eqref{def:varL}, and \eqref{eq:defF}, respectively, we have
    \[\variational^{\bar\normalL}=\mincurves^{\normalL}\geq
    \minarcs^{\normalL}=\variationalrelaxed^{\var\normalL}.\]
\end{lemma}
\begin{proof}
    That $\mincurves^{\normalL}\geq %
    \minarcs^{\normalL}$ follows from the definitions. %
    It follows immediately from the definitions  that $\variational^{\bar\normalL}=\mincurves^{\normalL} \geq\variationalrelaxed^{\var\normalL}$.  

    Let us now show that $\minarcs^{\normalL}\leq \variationalrelaxed^{\var{\normalL}}$. Take a contender $\gamma$ for \eqref{min:variationalrelaxed}. Let $\varepsilon>0$.
    For every $t$ for which $\gamma'(t)$ is defined, by Choquet's theorem we can find a probability measure $\bar\nu_t$ supported on $\bar\F(t,\gamma(t))$ such that $\bar\gamma'(t)=\int_{\bar\F(t,\gamma'(t))}v\,d\bar\nu_t(v)$. Given $\bar\nu_t$, by the Measurable Selection Theorem \cite{aubinfrankowska}, there is at least one measure $\nu_t$ with $\tilde f(t,x,\cdot)_{\#}\nu_t=\bar\nu_t$, and $\nu_t$ supported on $(t,x,u)$ with $\normalL(t,x,u)+\varepsilon\geq \var\normalL(t,x,v)$ for all $u\in\controls$ with $v=f(t,x,u)$. This gives a contender $(\gamma,(\nu_t)_t)$ for \eqref{min:relaxed} with cost at most $\varepsilon \T$ greater than that of $\gamma$. Since this is true for all $\varepsilon$ and every contender $\gamma$, we have $\minarcs^{\normalL}\leq \variationalrelaxed^{\var{\normalL}}$.

    Similarly, to show that $\minarcs^{\normalL}\geq \variationalrelaxed^{\var{\normalL}}$, we can take a contender $(\nu_t)_t$ for \eqref{min:relaxed} and apply Jensen's inequality to see that, if $\gamma(t)=\initial+\int_0^t\int_\controls f(s,x,u)\,d\nu_s(u)ds\in\domain$, then $\gamma'(t)=\int_\controls f(t,\gamma(t),u)\,d\nu_t(u)\in\operatorname{conv}\bar\F(t,\gamma(t))$  for almost every $t$, and
    \begin{equation*}
       \int_{0}^\T\int_\controls\normalL(t,\gamma(t),u)\,d\nu_t(u)\,dt\geq \int_0^\T\int_\controls \var\normalL(t,\gamma(t),f(t,\gamma(t),u))\,d\nu_t(u)\,dt %
       \geq \int_0^T\var\normalL(t,\gamma(t),\gamma'(t))\,dt. 
    \end{equation*}
    In other words, the $\normalL$-cost of $(\nu_t)$ is bounded from below by the $\var\normalL$-cost of $\gamma$, which is what we wanted to show.
\end{proof}

\printbibliography

@unpublished{OurSufficientPaper,
    author = {Augier, Nicolas and Korda, Milan and Rios-Zertuche, Rodolfo},
    title = {Sufficient conditions for the absence of relaxation gaps in state-constrained optimal control},
    note = {Preprint}
}

@article{augier2024symmetry,
  title={Symmetry reduction and recovery of trajectories of optimal control problems via measure relaxations},
  author={Augier, Nicolas and Henrion, Didier and Korda, Milan and Magron, Victor},
  journal={ESAIM: Control, Optimisation and Calculus of Variations},
  volume={30},
  pages={63},
  year={2024},
  publisher={EDP Sciences}
}

@article{korda2022gap,
  title={The gap between a variational problem and its occupation measure relaxation},
  author={Korda, Milan and Rios-Zertuche, Rodolfo},
  journal={arXiv preprint arXiv:2205.14132},
  year={2022}
}

@article {OCP08,
    AUTHOR = {Lasserre, Jean B. and Henrion, Didier and Prieur, Christophe
              and Tr\'{e}lat, Emmanuel},
     TITLE = {Nonlinear optimal control via occupation measures and
              {LMI}-relaxations},
   JOURNAL = {SIAM J. Control Optim.},
  FJOURNAL = {SIAM Journal on Control and Optimization},
    VOLUME = {47},
      YEAR = {2008},
    NUMBER = {4},
     PAGES = {1643--1666},
      ISSN = {0363-0129},
   MRCLASS = {49J15 (28A99 49J45 90C22 93C10)},
  MRNUMBER = {2421324},
MRREVIEWER = {Ilaria Fragal\`a},
       DOI = {10.1137/070685051},
       URL = {https://doi.org/10.1137/070685051},
}

@article{V93,
author = {Vinter, Richard},
title = {Convex Duality and Nonlinear Optimal Control},
journal = {SIAM Journal on Control and Optimization},
volume = {31},
number = {2},
pages = {518-538},
year = {1993},
doi = {10.1137/0331024},

URL = { 
    
        https://doi.org/10.1137/0331024
    
    

},
eprint = { 
    
        https://doi.org/10.1137/0331024
    
    

}
}

@article{jean2001uniform,
  title={Uniform estimation of sub-Riemannian balls},
  author={Jean, Fr{\'e}d{\'e}ric},
  journal={Journal of Dynamical and Control Systems},
  volume={7},
  number={4},
  pages={473--500},
  year={2001},
  publisher={Springer}
}

@book {ABB20,
    AUTHOR = {Agrachev, Andrei and Barilari, Davide and Boscain, Ugo},
     TITLE = {A comprehensive introduction to sub-{R}ie\-man\-nian geometry},
    SERIES = {Cambridge Studies in Advanced Mathematics},
    VOLUME = {181},
      NOTE = {},
 PUBLISHER = {Cambridge University Press, Cambridge},
      YEAR = {2020},
     PAGES = {},
      ISBN = {978-1-108-47635-5},
   MRCLASS = {53C17},
  MRNUMBER = {3971262},
MRREVIEWER = {Luca\ Rizzi},
}

@article{frankowska2000filippov,
  title={Filippov's and Filippov--Wa{\.z}ewski's Theorems on Closed Domains},
  author={Frankowska, H{\'e}l{\`e}ne and Rampazzo, Franco},
  journal={Journal of Differential Equations},
  volume={161},
  number={2},
  pages={449--478},
  year={2000},
  publisher={Elsevier}
}

@book{aubinfrankowska,
  title={Set-valued analysis},
  author={Aubin, Jean-Pierre and Frankowska, H{\'e}l{\`e}ne},
  year={1990},
  publisher={Birkh{\"{a}}user, Boston}
}

@article{palladino2014minimizers,
  title={Minimizers that are not also relaxed minimizers},
  author={Palladino, Michele and Vinter, Richard B},
  journal={SIAM Journal on Control and Optimization},
  volume={52},
  number={4},
  pages={2164--2179},
  year={2014},
  publisher={SIAM}
}

@article{jean2003entropy,
  title={Entropy and complexity of a path in sub-Riemannian geometry},
  author={Jean, Fr{\'e}d{\'e}ric},
  journal={ESAIM: Control, Optimisation and Calculus of Variations},
  volume={9},
  pages={485--508},
  year={2003},
  publisher={EDP Sciences}
}

@book{fulton2013algebraic,
  title={Algebraic topology: a first course},
  author={Fulton, William},
  volume={153},
  year={2013},
  publisher={Springer Science \& Business Media}
}

@incollection{korda2018moments,
  author      = "Korda, Milan and Henrion, Didier and Lasserre, Jean-Bernard",
  title       = "Moments and convex optimization for analysis and control of nonlinear {PDEs}",
  editor      = "E. Trelat, E. Zuazua",
  booktitle   = "Handbook of Numerical Analysis",
  publisher   = "Elsevier",
  year        = "2022",
  volume        = "23",
  pages       = "339--366",
  chapter     = 10,
}

@article {bettiolfrankowska,
    AUTHOR = {Bettiol, Piernicola and Frankowska, H\'el\`ene},
     TITLE = {Regularity of solution maps of differential inclusions under
              state constraints},
   JOURNAL = {Set-Valued Anal.},
  FJOURNAL = {Set-Valued Analysis. An International Journal Devoted to the
              Theory of Multifunctions and its Applications},
    VOLUME = {15},
      YEAR = {2007},
    NUMBER = {1},
     PAGES = {21--45},
      ISSN = {0927-6947,1572-932X},
   MRCLASS = {34A60 (49K24 49N60)},
  MRNUMBER = {2308710},
MRREVIEWER = {Aurelian\ Cernea},
       DOI = {10.1007/s11228-006-0018-4},
       URL = {https://doi.org/10.1007/s11228-006-0018-4},
}

@article {palais,
    AUTHOR = {Palais, Richard S.},
     TITLE = {Extending diffeomorphisms},
   JOURNAL = {Proc. Amer. Math. Soc.},
  FJOURNAL = {Proceedings of the American Mathematical Society},
    VOLUME = {11},
      YEAR = {1960},
     PAGES = {274--277},
      ISSN = {0002-9939,1088-6826},
   MRCLASS = {57.00},
  MRNUMBER = {117741},
MRREVIEWER = {A.\ M.\ Gleason},
       DOI = {10.2307/2032968},
       URL = {https://doi.org/10.2307/2032968},
}

@book {pedregal,
    AUTHOR = {Pedregal, Pablo},
     TITLE = {Parametrized measures and variational principles},
    SERIES = {Progress in Nonlinear Differential Equations and their
              Applications},
    VOLUME = {30},
 PUBLISHER = {Birkh\"auser Verlag, Basel},
      YEAR = {1997},
     PAGES = {xii+212},
      ISBN = {3-7643-5697-9},
   MRCLASS = {49-02 (35Q72 49J40 49Q20 73C50 73V25)},
  MRNUMBER = {1452107},
MRREVIEWER = {Tom\'a\v s\ Roub\'i\v cek},
       DOI = {10.1007/978-3-0348-8886-8},
       URL = {https://doi.org/10.1007/978-3-0348-8886-8},
}

@article {pedregal2,
    AUTHOR = {Pedregal, Pablo},
     TITLE = {Optimization, relaxation and {Y}oung measures},
   JOURNAL = {Bull. Amer. Math. Soc. (N.S.)},
  FJOURNAL = {American Mathematical Society. Bulletin. New Series},
    VOLUME = {36},
      YEAR = {1999},
    NUMBER = {1},
     PAGES = {27--58},
      ISSN = {0273-0979,1088-9485},
   MRCLASS = {49Q20 (49J10 74B20 74P99)},
  MRNUMBER = {1655480},
MRREVIEWER = {Martin\ Fuchs},
       DOI = {10.1090/S0273-0979-99-00774-0},
       URL = {https://doi.org/10.1090/S0273-0979-99-00774-0},
}

@book {vinter-book,
    AUTHOR = {Vinter, Richard},
     TITLE = {Optimal control},
    SERIES = {Modern Birkh\"auser Classics},
      NOTE = {Paperback reprint of the 2000 edition},
 PUBLISHER = {Birkh\"auser Boston, Ltd., Boston, MA},
      YEAR = {2010},
     PAGES = {xx+507},
      ISBN = {978-0-8176-4990-6},
   MRCLASS = {49-02 (49-01)},
  MRNUMBER = {2662630},
       DOI = {10.1007/978-0-8176-8086-2},
       URL = {https://doi.org/10.1007/978-0-8176-8086-2},
}

@article{WARGA197541,
author={Warga, J.},
title = {Necessary conditions without differentiability assumptions in optimal control},
journal = {Journal of Differential Equations},
volume = {18},
number = {1},
pages = {41-62},
year = {1975},
issn = {0022-0396},
doi = {https://doi.org/10.1016/0022-0396(75)90080-7},
}

@book{warga2014optimal,
  title={Optimal control of differential and functional equations},
  author={Warga, Jack},
  year={1972},
  publisher={Academic press}
}
\end{document}